\documentclass[12pt]{amsart}

\newtheorem{theorem}{Theorem}[section]
\newtheorem{lemma}[theorem]{Lemma}
\newtheorem{proposition}[theorem]{Proposition}
\newtheorem{corollary}[theorem]{Corollary}   
\newtheorem{definition}[theorem]{Definition}
\newtheorem{example}[theorem]{Example}

\numberwithin{equation}{section}

\usepackage[utf8]{inputenc}
\usepackage[english]{babel}
\usepackage{times}
\usepackage{enumerate}
\usepackage{tfrupee}
\usepackage{tikz}
\usetikzlibrary{chains,fit}
\usetikzlibrary{shapes,snakes}
\usetikzlibrary{graphs}
\usepackage{graphicx,adjustbox}
\usepackage{hyperref}
\usepackage{amsmath,amssymb}
\usepackage{amscd}
\usepackage[all,cmtip]{xy}
\usepackage{calrsfs} 
\usepackage{tabularx}
\usepackage{amsthm}
\usepackage{mathalfa,mathrsfs}
\usepackage[top=1in, bottom=1in, left=.9in, right=.9in]{geometry}
\usepackage[all]{xy}
\begin{document}

\title{On the algebraic invariants of certain affine semigroup algebras}

\address{IIT Gandhinagar, Palaj, Gandhinagar, Gujarat-382355 India}

\author{Om Prakash Bhardwaj}
\email{om.prakash@iitgn.ac.in}

\author{Indranath Sengupta}
\email{indranathsg@iitgn.ac.in} 

\thanks{The second author is the corresponding author.}

\maketitle

\begin{abstract}
Let $a$ and $d$ be two linearly independent vectors in $\mathbb{N}^2$, over the field of rational numbers. For a positive integer $k \geq 2$, consider the sequence $a, a+d, \ldots, a+kd$ such that the affine semigroup $S_{a,d,k} = \langle a, a+d, \ldots, a+kd \rangle$ is minimally generated by this sequence. We study the properties of affine semigroup algebra $k[S_{a,d,k}]$ associated to this semigroup. We prove that $k[S_{a,d,k}]$ is always Cohen-Macaulay and it is Gorenstein if and only if $k=2$. For $k=2,3,4$, we explicitly compute the syzygies, minimal graded free resolution and Hilbert series of $k[S_{a,d,k}].$ We also give a minimal generating set and a Gr\"{o}bner basis of the defining ideal of $k[S_{a,d,k}].$ Consequently, we prove that $k[S_{a,d,k}]$ is Koszul. Finally, we prove that the Castelnuovo-Mumford regularity of $k[S_{a,d,k}]$ is $1$ for any $a,d,k.$
\end{abstract}

\section{Introduction}
Throughout this article, the sets of integers and non-negative integers will be denoted by $\mathbb{Z}$ and $\mathbb{N}$ respectively. By an affine semigroup, we mean a finitely generated submonoid of $\mathbb{N}^r$ for some positive integer $r$. More precisely, let $a_1, \ldots , a_n \in \mathbb{N}^r$ then
\[ \langle a_1, \ldots ,a_n \rangle = \left\lbrace \sum_{i=1}^{n} \lambda_i a_i \mid \lambda_i \in \mathbb{N}, \forall i \right\rbrace\]
is called an affine semigroup generated by  $a_1, \ldots,  a_n$.  If $S:= \langle a_1, \ldots, a_n \rangle$ is generated by $ a_1, \ldots, a_n $ minimally then $\{a_1, \ldots, a_n \}$ is called the minimal system of generators of $S$. It is known that, an affine semigroup $S$ has a unique minimal set of generators. The cardinality of this set is called the embedding dimension of $S$. Let $k$ be a field. The semigroup ring $k[S]:=\oplus_{s \in S} k {\bf t}^s $ of $S$ is a $k$-subalgebra of the polynomial ring $k[t_1,\ldots,t_r]$, where $t_1,\ldots,t_r$ are indeterminates and ${\bf t}^s = \prod_{i=1}^r t_i^{s_i}$ for all $s = (s_1,\ldots,s_r) \in S$. The semigroup ring $k[S] = k[{\bf t}^{a_1}, \ldots, {\bf t}^{a_n}]$ of $S$ can be represented as a quotient of a polynomial ring using a canonical surjection $\pi : k[x_1,\ldots,x_n] \rightarrow k[S]$ given by $\pi(x_i) = {\bf t}^{a_i}$ for all $i=1,\ldots,n.$ 
Set $\deg x_i = a_i$ for all $i=1,\ldots,n.$ Observe that $k[x_1,\ldots,x_n]$ is a multi-graded ring and that $\pi$ is a degree preserving surjective $k$-algebra homomorphism. We denote by $I_S$ the kernel of $\pi.$ Then $I_S$ is a homogeneous ideal, generated by binomials, called the defining ideal of $S$. Note that a binomial $\phi = \prod_{i=1}^n x_i^{\alpha_i} - \prod_{j=1}^n x_j^{\beta_j} \in I_S$ if and only if $\sum_{i=1}^n \alpha_i a_i = \sum_{j=1}^n \beta_j a_j.$ With respect to this grading, $\deg_S \phi = \sum_{i=1}^n \alpha_i a_i.$ For $r =1$, the affine semigroups correspond to the numerical semigroups. A numerical semigroup $S$ is a submonoid of $\mathbb{N}$ such that $\mathbb{N} \setminus S$ is finite. Equivalently, there exist $a_1,\ldots,a_n \in \mathbb{N}$ such that $\mathrm{gcd}(a_1,\ldots,a_n) = 1$ and 
\[
S =  \langle a_1, \ldots, a_n \rangle = \left\lbrace \sum_{i=1}^{n} \lambda_i a_i \mid \lambda_i \in \mathbb{N}, \forall i \right\rbrace.
\]

The numerical semigroup ring $k[S]$ has been studied widely in the literature. A special class of numerical semigroup rings associated to the numerical semigroup generated by an arithmetic sequence is well studied, for example in (\cite{bermejo-isabel}, \cite{sengupta-srinivasan},  \cite{maloo-sengupta}, \cite{patil-sengupta}, \cite{patil-singh}) and many more. This motivates us to study a similar case in higher dimension. Let $a,d \in \mathbb{N}^2$ be two linearly independent vectors over the field of rational numbers $\mathbb{Q}$. In this article, for $k \in \mathbb{N}$, we consider the sequence $a, a+d, \ldots, a+kd$ such that the affine semigroup $S_{a,d,k} := \langle a, a+d, \ldots, a+kd \rangle$ is minimally generated by this sequence. Therefore, $S_{a,d,k}$ is an affine semigroup of embedding dimension $k+1$. 

Now, we summarize the contents of the paper. Section 2 recalls some definitions and results about simplicial affine semigroups. In \cite{jafari-2}, authors introduce the notion of quasi-Frobenius elements of a simplicial affine semigroup $S$. The set of quasi-Frobenius elements is denoted by $\mathrm{QF}(S)$. They also prove that if the semigroup ring $k[S]$ associated to a simplicial affine semigroup $S$ is Cohen-Macaulay then the Cohen-Macaulay type of $k[S]$ is equal to the cardinality of the set $\mathrm{QF}(S)$. In section 3, we explore the Cohen-Macaulayness and Gorenstein properties of $k[S_{a,d,k}]$. In Theorem \ref{CohenMacaulayness}, we prove that $k[S_{a,d,k}]$ is Cohen-Macaulay. In Theorem \ref{typeSad}, we explicitly compute the set $\mathrm{QF}(S_{a,d,k})$ and consequently, we prove that the Cohen-Macaulay type of $k[S_{a,d,k}]$ is $k-1$. 
In section 4, we give a minimal generating set of the defining ideal of $k[S_{a,d,k}]$. In Theorem \ref{k-dim J suset I_S}, we give an extension of the Gastinger's theorem (\cite{gastinger}, \cite[Theorem 4.8]{etobalancesemigroup}) to the simplicial affine semigroups. This is the key result of the section. With the help of Theorem \ref{k-dim J suset I_S}, we give a minimal generating set of the defining ideal $I_{S_{a,d,k}}$ of the semigroup ring $k[S_{a,d,k}]$ in Theorem \ref{generatingset}. In Theorem \ref{Groebnerbasis}, we prove that the proposed generating set is actually a Gr\"{o}bner basis for $I_{S_{a,d,k}}$ with respect to the reverse lexicographic order. In section 5, we explore the syzygies of $k[S_{a,k,d}]$ for $k=2,3,4.$ In Proposition \ref{resk=3} and Proposition \ref{hilbertseriesk=3}, we compute the minimal graded free resolution and the Hilbert series of $k[S_{a,d,k}]$ when $k=3$. In Proposition \ref{resk=4} and Proposition \ref{hilbertseriesk=4}, we compute the minimal graded free resolution and the Hilbert series of $k[S_{a,d,k}]$ when $k=4$. For a graded ideal $I$ of $k[x_1,\ldots,x_n]$, the ring $k[x_1,\ldots,x_n]/I$ is called Koszul if the minimal free resolution of $k$ over $k[x_1,\ldots,x_n]/I$ is linear. Finding the classes of Koszul rings is an interesting question in commutative algebra. In Theorem \ref{Koszul}, we prove that $k[S_{a,d,k}]$ is Koszul. By the Theorem \ref{generatingset}, we observe that $I_{S_{a,k,d}}$ is homogeneous with respect to the standard grading also on $k[x_1,\ldots,x_{k+1}]$. In section 6, Theorem \ref{regI_S_akd}, we prove that the Castelnuovo-Mumford regularity of $I_{S_{a,k,d}}$ is $2$ for any $a,k,d$. In section 7, we give an extension of $k[S_{a,d,k}]$ with respect to an element $b \in \mathbb{N}^2$. Define $S_{a,d,k}^b:= \langle a, a+d, \ldots, a+kd, b \rangle$. Let $R=k[x_1,\ldots,x_{k+1},y]$ and $\phi: k[x_1,\ldots,x_{k+1},y] \rightarrow k[t_1,t_2]$ such that $x_i \mapsto {\bf t}^{a+(i-1)d}$ for $i=1,\ldots,k+1$, $y \mapsto {\bf t}^b$. The semigroup ring $k[S_{a,d,k}^b]$ is isomorphic to the quotient ring $\frac{R}{I_{S_{a,d,k}^b}}$, where $I_{S_{a,d,k}^b}$ is the kernel of the algebra homomorphism $\phi$. With some conditions on $b$, we explicitly compute the set $\mathrm{QF}(S_{a,d,k}^b)$ and prove that $k[S_{a,d,k}^b]$ is Cohen-Macaulay (Proposition \ref{Aperysetb}, Corollary \ref{QF(S)b}). Consequently, we prove that the Cohen-Macaulay type of $k[S_{a,d,k}^b]$ is equal to the Cohen-Macaulay type of $k[S_{a,d,k}]$. We conclude the article by showing the utility of our results with a couple of examples in section 8.
\section{Preleminaries}

\begin{definition}
The affine semigroup $S= \langle a_1,\ldots,a_n\rangle \subset \mathbb{N}^r$ is called simplicial if there exist $a_{i_1},\ldots, a_{i_r} \in \{ a_1,\ldots,a_n\}$ such that\\
$(1)$ $a_{i_1},\ldots, a_{i_r}$ are linearly independent over $\mathbb{Q}$\\
$(2)$ for each $a \in S$, there exist $0 \neq \alpha \in \mathbb{N}$ such that $\alpha a = \sum_{j=1}^r \lambda_j a_{i_j}$, where $\lambda_j \in \mathbb{N}$.
\end{definition}

The elements $a_{i_1},\ldots, a_{i_r}$ are known as the extremal rays of $S$, which are the componentwise smallest elements of $S$ on the each extremal ray of the cone generated by $S$, where
\[
\mathrm{cone(S)} := \left\lbrace \sum_{i=1}^{n} \lambda_i a_i \mid \lambda_i \in \mathbb{R}_{\geq 0}, \forall i \right\rbrace.
\]
Let $S$ be a simplicial affine semigroup, given any element $b \in \mathbb{N}^r \cap \mathrm{cone}(S)$, there exist $\lambda_j \in \mathbb{Q}_{\geq 0}$ such that $b = \sum_{j=1}^r \lambda_j a_{i_j}$. We denote $\mathrm{deg}(b)$ by $\sum_{j=1}^r \lambda_j$.
\begin{definition}
For an element $0 \neq a \in S$, the Apery set of $a$ is defined as 
\[  
\mathrm{Ap}(S,a)= \{b \in S \mid b-a \notin S\},
\]
 and for a subset $E$ of $S,$ the Apery set of $E$ is defined as 
\[
\mathrm{Ap}(S,E)= \{b \in S \mid b-a \notin S, \forall a \in E\}.
\]
\end{definition}

Throughout this article $S = \langle a_1,\ldots,a_r,a_{r+1},\ldots, a_n \rangle $ will denote a simiplicial affine semigroup in $\mathbb{N}^r,$ where $\{a_1,\ldots,a_r\}$ are the extremal rays of $S.$ Consider the partial order $\preceq_S$ on $\mathbb{N}^r,$ where for all $x,y \in \mathbb{N}^r,~ x \preceq_S y$ if $y - x \in S.$ Denote $E = \{a_1,\ldots,a_r\}.$ In this setup, the authors in \cite{jafari-2} define the set of quasi-Frobenius elements of $S.$

\begin{definition}
For any $b \in \max_{\preceq} \mathrm{Ap}(S,E),$ the element $b - \sum_{i=1}^r a_i$ is called a quasi-Frobenius element of $S.$ The set of all quasi-Frobenius elements is denoted by $\mathrm{QF}(S).$
\end{definition}

\begin{proposition}(\cite[Proposition 3.3]{jafari-2})\label{type}
Let $S$ be a simplicial affine semigroup. If $k[S]$ is Cohen-Macaulay then the Cohen-Macaulay type of $k[S]$ is equal to the cardinality of the set $\mathrm{QF}(S).$
\end{proposition}

Let $a = (a_1,a_2)$ and $d = (d_1,d_2) \in \mathbb{N}^2$ such that these are $\mathbb{Q}$-linearly independent. For $k \geq 1,$ Define 
\begin{center}
$S_{a,d,k} = \langle a, a+d, a+2d, \ldots a+kd \rangle .$
\end{center}
We always assume that $\{a, a+d, a+2d, \ldots a+kd \}$ is a minimal generating set of $S_{a,d,k}.$ Note that for $k=1$, the semigroup ring $k[S_{a,d,k}]$ is isomorphic to the polynomial ring $k[x_1,x_2]$. Therefore we will assume that $k \geq 2$.

\section{The set  $\mathrm{QF}(S_{a,d,k})$}

\begin{lemma}\label{simplicial}
$S_{a,d,k}$ is a simplicial affine semigroup in $\mathbb{N}^2$ with respect to the extremal rays $a$ and $a+kd.$
\end{lemma}
\begin{proof}
It is clear from the following equations:
\[
k(a+d) = (k-1)a + (a+kd)
\]
\[
k(a+2d) = (k-2)a + 2(a+kd)
\]
\[
\vdots \qquad\qquad\qquad\qquad \vdots 
\]
\[
k(a+(k-1)d) = a + (k-1)(a+kd)
\]
\end{proof}

\begin{lemma}\label{Aperyset}
Let $E = \{a, a+kd\}.$ The Apery set of $S_{a,d,k}$ with respect to $E$ is 
\[
\mathrm{Ap}(S_{a,d,k},E) = \{0,a+d, a+2d, \ldots, a+(k-1)d\}. 
\]
\end{lemma}
\begin{proof}
Observe that
\[
\mathrm{Ap}(S_{a,d,k},E) \subset \left\lbrace \sum_{i=1}^{k-1} \lambda_i (a+id) \mid \lambda_i \in \mathbb{N} \right\rbrace .
\]
Let $0 \neq b \in \mathrm{Ap}(S_{a,d,k},E),$ therefore for some $\lambda_i 's \in \mathbb{N},$ we can write $b = \sum_{i=1}^{k-1} \lambda_i (a+id).$ We claim that $\lambda_i \in \{0,1\}$ for all $i \in [1,k-1].$ Suppose there exist $\ell \in [1,k-1]$ such that $\lambda_{\ell} > 1.$ Therefore, we can write $b$ as
\[
b = 2 (a + \ell d) + (\lambda_{\ell} - 2)(a + \ell d) + \sum_{i=1, i \neq \ell}^{k-1} \lambda_i (a+id).
\]  
\textbf{Case-1:} Suppose $2 \ell \leq k.$ Therefore we have 
$2 (a + \ell d) - a = a + 2 \ell d \in S_{a,d,k},$ and hence
\[
b-a = (a + 2 \ell d) + (\lambda_{\ell} - 2)(a + \ell d) + \sum_{i=1, i \neq \ell}^{k-1} \lambda_i (a+id) \in S_{a,d,k}.
\]
Which is a contradiction to $b \in \mathrm{Ap}(S_{a,d,k},E).$\\
\textbf{Case-2:} Suppose $2 \ell > k.$ We can write $2 \ell = k + k',$ where $k' \in [0,k-2].$ We can write 
\[
b =  (a + k d) + (a + k' d ) + (\lambda_{\ell} - 2)(a + \ell d) + \sum_{i=1, i \neq \ell}^{k-1} \lambda_i (a+id).
\]
Therefore, we have
\[
b - (a+kd) = (a + k' d ) + (\lambda_{\ell} - 2)(a + \ell d) + \sum_{i=1, i \neq \ell}^{k-1} \lambda_i (a+id) \in S_{a,d,k}.
\]
Which is again a contradiction to $b \in \mathrm{Ap}(S_{a,d,k},E).$ Therefore we conclude that 
\[
\mathrm{Ap}(S_{a,d,k},E) \subseteq \left\lbrace \sum_{i=1}^{k-1} \lambda_i (a+id) \mid \lambda_i \in \{0,1\} \right\rbrace .
\]

Now we claim that $b = a+id$ for some $i \in [1,k-1].$ Suppose  there exist $\ell , \ell ' \in [1,k-1]$ such that $\ell \neq \ell '$ and $\lambda_{\ell} = 1 = \lambda_{\ell '}.$ Therefore we can write
\[
b =  (a + \ell d) + (a + \ell ' d ) + \sum_{i=1, i \neq \ell , \ell '}^{k-1} \lambda_i (a+id), \quad \mathrm{where}~\lambda_i \in \{0,1\}.
\]
\textbf{Case-1:} Suppose $\ell + \ell ' \leq k.$ We have $(a + \ell d) + (a + \ell ' d ) - a = a + (\ell + \ell ') d \in S.$ Therefore, we have 
\[
b-a =  a + (\ell + \ell ') d + \sum_{i=1, i \neq \ell , \ell '}^{k-1} \lambda_i (a+id), \quad \mathrm{where}~\lambda_i \in \{0,1\}.
\]
Hence, we get $b-a \in S_{a,d,k}.$ Which is contradiction to $b \in \mathrm{Ap}(S_{a,d,k},E).$\\
\textbf{Case-2:} Suppose $\ell + \ell ' > k.$ We can write $\ell + \ell ' = k+ \ell'',$ where $\ell'' \in [0,k-3].$ Therefore, we can write $b$ as 
\[
b =  (a + k d) + (a + \ell '' d ) + \sum_{i=1, i \neq \ell , \ell '}^{k-1} \lambda_i (a+id), \quad \mathrm{where}~\lambda_i \in \{0,1\}.
\]
Therefore, we have 
\[
b-(a+kd) =  (a + \ell '' d ) + \sum_{i=1, i \neq \ell , \ell '}^{k-1} \lambda_i (a+id) \in S_{a,d,k}; \quad \mathrm{where}~\lambda_i \in \{0,1\}.
\]
Which is again a contradiction to $b \in \mathrm{Ap}(S_{a,d,k},E).$ Hence the claim follows. Therefore, we conclude that
\[
\mathrm{Ap}(S_{a,d,k},E) \subseteq \{0, a+d, a+2d, \ldots, a+(k-1)d\}. 
\]
Conversely, suppose $b \in \{a+d, a+2d, \ldots, a+(k-1)d\}.$ Therefore $b = a + \ell d$ for some $\ell \in [1,k-1].$ Since $\ell < k,$ It is clear that $b - (a + kd) = (\ell-k)d \notin S_{a,d,k}.$ Hence $b \in \mathrm{Ap}(S_{a,d}, a+kd).$ Also $b -a \notin S_{a,d,k}.$ Suppose $b - a = \ell d \in S_{a,d,k}.$ Since $\ell d \leq a+ \ell d$ (componentwise) and $\ell d \in S_{a,d,k}$, we can write 
\[
\ell d = \sum_{j=1}^{\ell - 1} \lambda_j (a+id); \quad \mathrm{where} ~ \lambda_j \in \mathbb{N}.
\]
Therefore, we have 

\[
b = a + \ell d = (\lambda_0 + 1)a + \sum_{j=1}^{\ell - 1} \lambda_j (a+jd); \quad \mathrm{where} ~ \lambda_j \in \mathbb{N}.
\]
Since $\ell \in [1,k-1]$, we get a contradiction to the minimality of the generating set of $S_{a,d,k}.$ Hence $b - a \notin S_{a,d,k},$ and therefore $b \in \mathrm{Ap}(S_{a,d,k}, a).$ This completes the proof.
\end{proof}

The following theorem gives a useful criterion for the Cohen-Macaulayness of a simplicial affine semigroup.

\begin{theorem}\label{rosalesoncm}(\cite[Cor. 1.6]{rosalesOnCohen-Macaulayness})
Let $S = \langle a_1,\ldots,a_r,a_{r+1},\ldots, a_n \rangle$ be a simiplicial affine semigroup in $\mathbb{N}^r,$ where $\{a_1,\ldots,a_r\}$ are the extremal rays of $S.$ The following statements are equivalent 
\begin{enumerate}[{\rm(i)}]
		\item $k[S]$ is Cohen-Macaulay,
		\item for all $x,y \in \cap_{i=1}^r \mathrm{Ap}(S,a_i),$ if $x \neq y,$ then $x - y \notin G(\{a_1,\ldots,a_r \}),$
	\end{enumerate}
where $G(\{a_1,\ldots,a_r \})$ denote the group generated by $a_1,\ldots,a_r$ in $\mathbb{Z}^r$.
\end{theorem}

\begin{theorem}\label{CohenMacaulayness}
$k[S_{a,d,k}]$ is Cohen-Macaulay.
\end{theorem}
\begin{proof}
let $b , b' \in \mathrm{Ap}(S_{a,d,k},E),$ where $E = \{ a, a+d\}$ and $b \neq b'.$ Suppose $b = a + \ell d$ and $b' = a+ \ell ' d$, where $l,l' \in [1,k-1].$ Observe that $b-b' = (\ell - \ell ')d \notin G(\{a, a+kd \}).$ Suppose $b-b' \in G(\{a, a+kd \}),$ there exist $s_1,s_2 \in \langle a , a+kd \rangle$ such that $b-b' = (\ell - \ell ')d = s_1-s_2.$ Therefore, for some $\lambda_1, \lambda_2, \lambda_1', \lambda_2' \in \mathbb{N},$  we get
\begin{align*}
(\ell - \ell ')d &= \lambda_1 a + \lambda_2 (a+kd) - \lambda_1' a - \lambda_2' (a+kd)\\
&= (\lambda_1 + \lambda_2 - \lambda_1' - \lambda_2') a + (\lambda_2 - \lambda_2')k d.
\end{align*}

Since $a$ and $d$ are linearly independent over $\mathbb{Q}$, we have $|(\lambda_2 - \lambda_2')k| = |\ell - \ell '|.$ Which is absurd, since $0 < |\ell - \ell '| < k$ but $|(\lambda_2 - \lambda_2')k|$ is either zero or $|(\lambda_2 - \lambda_2')k| \geq k.$ Also, note that for any $b \in \mathrm{Ap}(S_{a,d,k},E)$, $b \notin G(\{a, a+kd \}).$ Hence by Theorem \ref{rosalesoncm}, $k[S_{a,d,k}]$ is Cohen-Macaulay.
\end{proof}

\begin{theorem}\label{typeSad}
$\mathrm{QF}(S_{a,d,k}) = \{ -(a+d), -(a+2d), \ldots, -(a+(k-1)d)\}.$ In particular, Cohen-Macaulay type of $k[S_{a,d,k}]$ is $k-1.$
\end{theorem}
\begin{proof}
From lemma \ref{Aperyset}, we see that
\[
\max_{\preceq} \mathrm{Ap}(S_{a,d,k},E) = \{a+d, a+2d, \ldots, a+(k-1)d \}
\]
Therefore by lemma \ref{simplicial}, we have
\[
\mathrm{QF}(S_{a,d,k}) = \{ -(a+d), -(a+2d), \ldots, -(a+(k-1)d)\}.
\]
In particular, by Theorem \ref{type}, Cohen-Macaulay type of $k[S_{a,d,k}]$ is equal to $|\mathrm{QF}(S_{a,d,k})| = k-1.$
\end{proof}

\begin{corollary}\label{Gorenstein}
$k[S_{a,d,k}]$ is Gorenstein if and only if $k=2.$
\end{corollary}
\begin{proof}
If $k=2,$ then by Theorem  \ref{typeSad}, we have $\mathrm{type}~ k[S_{a,d,k}] = 1$ and hence $k[S_{a,d,k}]$ is Gorenstein. Conversely, if $k[S_{a,d,k}]$ is Gorenstein then $\mathrm{type}~ k[S_{a,d,k}] = 1.$ By Theorem  \ref{typeSad}, we have $k=2.$ 
\end{proof}

An affine semigroup $S$ is called normal if $S = \bar{S}$, where
\[
\bar{S} = \{a \in G(S) \mid \alpha a\in S ~\mathrm{for~some}~ \alpha \in \mathbb{N}\} = \mathrm{cone}(S) \cap G(S).
\]
$\bar{S}$ is called normalization of $S$. The affine semigroup ring $k[S]$ is normal ring if and only if $S$ is normal affine semigroup. In \cite[Theorem 4.6]{jafari-2}, authors gives a criterion for a simplicial affine semigroup $S$ being normal in terms of quasi-Frobenius elements and the relative interior of the $\mathrm{cone}(S).$ The relative interior of $\mathrm{cone}(S)$ is defined as
\[
\mathrm{relint(S)} := \left\lbrace a \in \mathrm{cone}(S) \mid a = \sum_{i=1}^{r} \lambda_i a_i~ \mathrm{with} ~ \lambda_i \in \mathbb{R}_{> 0}, \forall i = 1,\ldots,r \right\rbrace.
\]


\begin{corollary}\label{normal}
$k[S_{a,d,k}]$ is a normal ring.
\end{corollary}
\begin{proof}
From \cite[Theorem 4.6]{jafari-2}, a simplicial affine semigroup $S$ is normal if and only if $- ~ \mathrm{QF}(S) \subseteq \mathrm{relint}(S).$ For $1 \leq \ell \leq k-1,$ we have $a + \ell d = \frac{\ell}{k}(a +kd) + (1 - \frac{\ell}{k})a.$ Therefore by Theorem \ref{typeSad}, we have $- ~ \mathrm{QF}(S_{a,d,k}) \subseteq \mathrm{relint}(S_{a,d,k}).$ 
Hence $S_{a,d,k}$ is normal.
\end{proof}

For an affine semigroup $S$, the semigroup ring $k[S]$ has a unique monomial maximal ideal, let us denote it by $\mathfrak{m}$. The associated graded ring of $k[S]$ is defined as
\[
\mathrm{gr}_{\mathfrak{m}}(k[S]) = \bigoplus_{i=0}^{\infty} \frac{{\mathfrak{m}^i}}{{\mathfrak{m}}^{i+1}}
\]

\begin{corollary}
The associated graded ring $\mathrm{gr}_{\mathfrak{m}}(k[S_{a,d,k}])$ is isomorphic to $k[S_{a,d,k}]$.
\end{corollary}
\begin{proof}
Observe that $a+id = (\frac{k-i}{k})a + \frac{i}{k}(a+kd)$ for all $i=1,\ldots,k-1$. Since $a$ and $a+kd$ are the extremal rays of $S_{a,d,k}$, we have $\mathrm{deg}(a+id) = 1$ for all $i=1,\ldots,k-1$. Now the result follows from the \cite[Remark 1.1]{jafaritangentcone}.
\end{proof}

\section{Minimal generating set for $I_{S_{a,d,k}}$}
Throughout the section $S = \langle a_1,\ldots,a_r,a_{r+1},\ldots, a_n \rangle $ will denote a simiplicial affine semigroup in $\mathbb{N}^r$ and  $ E =\{a_1,\ldots,a_r\}$ will denote the set of extremal rays of $S.$ Since, for each elements $a(\neq 0) \in S,$ we have $\alpha a \in \langle a_1,\ldots,a_r \rangle$ for some positive integer $\alpha$. Therefore $\{{\bf t}^{a_1}, \ldots, {\bf t}^{a_r}\}$ forms a system of parameter for $k[S]$.
\begin{lemma}\label{k-dim I_S} 
Let $S = \langle a_1, \ldots , a_r,a_{r+1}, \ldots, a_n \rangle \subset \mathbb{N}^r$ be a simplicial affine semigroup such that $ a_1, \ldots , a_r$ are the extremal rays of $S$. Then
\[
\dim_k \frac{k[x_1,\ldots,x_n]}{I_S + \langle x_1,\ldots,x_r \rangle} = |\mathrm{Ap}(S,E)|
\]
\end{lemma}
\begin{proof}
Follows from \cite[Theorem 3.3]{ojeda-tenorio}.
\end{proof}

\begin{theorem}\label{k-dim J suset I_S}
Let $S = \langle a_1, \ldots , a_r,a_{r+1}, \ldots, a_n \rangle \subset \mathbb{N}^r$ be a simplicial Cohen-Macaulay affine semigroup and $J$ be an ideal of $k[x_1,\ldots,x_n]$ such that $J \subset I_S$. If 
\[
\dim_k \frac{k[x_1,\ldots,x_n]}{J + \langle x_1,\ldots,x_r \rangle} = |\mathrm{Ap}(S,E)|
\]
then $J = I_S$.
\end{theorem}
\begin{proof}
Consider the exact sequence
\[
0 \rightarrow I_S/J \rightarrow k[x_1,\ldots,x_n]/J \rightarrow k[x_1,\ldots,x_n]/I_S \rightarrow 0
\]
Tensoring this sequence with $k[x_1,\ldots,x_n]/\langle x_1,\ldots,x_r \rangle$, we get long exact sequence of Tor modules
\begin{center}
$\cdots \rightarrow \mathrm{Tor}_1\left(\frac{k[x_1,\ldots,x_n]}{I_S},\frac{k[x_1,\ldots,x_n]}{ \langle x_1,\ldots,x_r \rangle}\right) \rightarrow  \frac{I_S/J}{\langle x_1,\ldots,x_r \rangle I_S/J} \rightarrow \frac{k[x_1,\ldots,x_n]}{J + \langle x_1,\ldots,x_r \rangle} \rightarrow \frac{k[x_1,\ldots,x_n]}{I_S + \langle x_1,\ldots,x_r \rangle} \rightarrow 0$
\end{center}
Since $\mathrm{dim}_k \frac{k[x_1,\ldots,x_n]}{J + \langle x_1,\ldots,x_r \rangle} = |\mathrm{Ap}(S,E)|$, we have $ \frac{k[x_1,\ldots,x_n]}{J + \langle x_1,\ldots,x_r \rangle} = \frac{k[x_1,\ldots,x_n]}{I_S + \langle x_1,\ldots,x_r \rangle}$ by Lemma \ref{k-dim I_S}. Since $S$ is Cohen-Macaulay and $\{x_1,\ldots,x_r\}$ is a regular sequence on $k[x_1,\ldots,x_n]/I_S$, we have $\mathrm{Tor}_1(\frac{k[x_1,\ldots,x_n]}{I_S},$\\ $\frac{k[x_1,\ldots,x_n]}{ \langle x_1,\ldots,x_r \rangle}) = 0$. Hence from the long exact sequence, we have $\frac{I_S/J}{\langle x_1,\ldots,x_r \rangle I_S/J} = 0.$ Therefore by Nakayama's lemma, we conclude that $I_S = J$.
\end{proof}

%
%
%
%



\textbf{Notation:} For each $\ell \in [2,k]$, we define sets $\xi_{\ell}$ of binomials.
\begin{itemize}
\item If $2 \ell > k+1$ then
\begin{center}
$\xi_{\ell} = \{x_{\ell}^2 - x_{2 \ell -k-1} x_{k+1} \} \cup \{x_{\ell}x_{\ell + i}-x_{2 \ell - k - 1 + i} x_{k+1} \mid 1 \leq i \leq k-\ell \}$
\end{center}

\item If $2 \ell \leq k+1$ then
\begin{center}
$\xi_{\ell} = \{x_{\ell}^2 - x_1 x_{2 \ell-1}  \} \cup \{x_{\ell}x_{\ell + i}-x_1x_{2 \ell- 1 + i} \mid 1 \leq i \leq k-2 \ell + 2 \} $\\
$ \cup  \{ x_{\ell}x_{\ell + i} - x_{2 \ell -k -1 + i}x_{k+1} \mid k- 2 \ell +2 < i \leq k- \ell \}$
\end{center}
\end{itemize}

\begin{lemma}\label{mingenlemma}
Let $J$ be the ideal generated by $G := \cup_{\ell = 2}^k \xi_{\ell}$. Then $G$ forms a minimal generating set for $J$.
\end{lemma}
\begin{proof}
Let $\phi : k[x_1,\ldots,x_{k+1}] \longrightarrow k[x_2,\ldots,x_k]$ be the map such that $\phi(x_i) = x_i$ for $i=2,\ldots,k$ and $\phi(x_1) = \phi(x_{k+1}) = 0$. Consider an arbitrary element $f \in G$, then $f = x_mx_n - x_{m'}x_{n'}$, where $m,n \in [2,k]$ and $m',n' \in [1,k+1]$. We also have that either $m' = 1$ or $n' = k+1$ for any $f \in G$. Suppose that 
\[
f = \sum_{g \in G, g \neq f} f'(x_1,\ldots,x_{k+1}) g.
\]
Now operating $\phi$ on both sides, we have
\[
x_mx_n = \sum f'(0,x_2,\ldots,x_k,0)x_{\ell}x_{\ell'} \quad \text{for some} \quad \ell, \ell' \in [2,k].
\]
This is not possible, since $g \neq f$. Hence $G$ forms a minimal generating set for $J$.
\end{proof}

\begin{theorem}\label{generatingset}
The set $G := \cup_{\ell = 2}^k \xi_{\ell}$ is a minimal generating set for the defining ideal $I_{S_{a,d,k}}$ of $k[S_{a,d,k}].$
\end{theorem}
\begin{proof}
Let $J$ be the ideal generated by $G$ in $k[x_1, \ldots, x_{k+1}].$ We have 
\begin{align*}
J + \langle x_1, x_{k+1}\rangle &= \langle \cup_{\ell = 2}^k \xi_{\ell} \rangle + \langle x_1, x_{k+1}\rangle \\
&=\langle \cup_{j=3}^k \{x_\ell x_j \mid 2 \leq \ell \leq j-1 \} \cup \{ x_{\ell}^2 \mid 2 \leq \ell \leq k \} \rangle + \langle x_1, x_{k+1}\rangle
\end{align*}

Define
 
\begin{center}
$J' = \langle \cup_{j=3}^k \{x_\ell x_j \mid 2 \leq \ell \leq j-1 \} \cup \{ x_{\ell}^2 \mid 2 \leq \ell \leq k \} \rangle .$
\end{center}
Therefore, we have
\begin{align*}
\frac{k[x_1, \ldots , x_{k+1}]}{J + \langle x_1, x_{k+1}\rangle} = \frac{k[x_2, \ldots, x_k]}{J'}
\end{align*}
  It is clear that the $k$-vector space $k[x_2, \ldots, x_k]/ J'$ is generated by the set $\mathfrak{B} = \{1, x_2, x_3, \ldots, x_k\}.$ Note that the cardinality of $\mathfrak{B}$ is $k.$ Also from Lemma \ref{Aperyset}, we see that the cardinality of $\mathrm{Ap}(S_{a,d,k},$\\ $ \{a, a+ kd\})$ is $k.$ Therefore by Theorem \ref{k-dim J suset I_S}, we conclude that $J = I_{S_{a,d,k}}.$ The minimality of the generating set follows from the Lemma \ref{mingenlemma}.
\end{proof}

\begin{corollary}
Let $\mu(I_{S_{a,d,k}})$ denote the cardinality of a minimal generating set of $I_{S_{a,d,k}}$. Then $\mu(I_{S_{a,d,k}}) = \frac{k(k-1)}{2}$.
\end{corollary}
\begin{proof}
Since $G := \cup_{\ell = 2}^k \xi_{\ell}$ in \ref{generatingset} is a minimal generating set of  $I_{S_{a,d,k}}$. Observe that for all $\ell \in [0,k-2]$, we have $|\xi_{k-\ell}| = \ell + 1.$ Hence,
\[
\mu(I_{S_{a,d,k}}) = \sum_{\ell = 0}^{k-2} |\xi_{k-\ell}| = \sum_{\ell = 0}^{k-2} \ell + 1 = \dfrac{k(k-1)}{2}.
\]
\end{proof}

Fix a monomial order. Let $\mathrm{LT}(f)$ denote the leading term of the polynomial $f$. Let $I$ be an ideal of $k[x_1,\ldots,x_n]$. A finite subset $G = \{g_1,\ldots,g_t\} \subset I$ is said to be a 
Gr\"{o}bner basis of $I$ if 
\[
\langle \mathrm{LT}(g_1), \ldots, \mathrm{LT}(g_t)\rangle = \langle \mathrm{LT}(I)\rangle.
\]
Following is the criterion of Buchberger for when a basis of an ideal is a Gr\"{o}bner basis.
\begin{theorem}[{\bf (Buchberger's criterion)}, {\cite[2.6, Theorem 6]{coxlittleoshea}}]  \label{Buchberger's criterion}
 Let $I$ be an ideal of $k[x_1,\ldots,x_n]$. Then a basis $G = \{g_1,\ldots,g_t\}$ for $I$ is a Gr\"{o}bner basis for $I$ if and only if for all pairs $i \neq j,$ the remainder on division of $S(g_i,g_j)$ by $G$ is zero. Where
\[
 S(g_i,g_j) = \frac{\bf x^{\gamma}}{\mathrm{LT}(g_i)} g_i - \frac{\bf x^{\gamma}}{\mathrm{LT}(g_j)} g_j
\]
with $\gamma = (\gamma_1, \ldots, \gamma_n)$ such that $\gamma_i = \mathrm{max}(\alpha_i,\beta_i)$ and $\mathrm{deg}(g_i)= (\alpha_1,\ldots, \alpha_n)$, $\mathrm{deg}(g_j) = (\beta_1,\ldots,\beta_n)$.
\end{theorem}
In the following Theorem we prove that $G$ defined in Theorem \ref{generatingset} is actually a Gr\"{o}bner basis for $I_{S_{a,d,k}}$ with respect to the reverse lexicographic order.
\begin{theorem}\label{Groebnerbasis}
With respect to the reverse lexicographic order, $G := \cup_{\ell = 2}^k \xi_{\ell}$ forms a Gr\"{o}bner basis for $I_{S_{a,d,k}}$.
\end{theorem}
\begin{proof}
First, we divide $G$ into following parts:
\begin{itemize}

\item $B_1 = \{ x_{\ell}^2 - x_1 x_{2 \ell -1} \mid \ell \in [2,k], 2\ell \leq k+1 \}$

\item $B_2 = \{ x_{\ell}x_{\ell + i} - x_1 x_{2 \ell -1 + i} \mid \ell \in [2,k], 2\ell \leq k+1, 1 \leq i \leq k -2 \ell + 2\}$

\item $B_3 = \{ x_{\ell}x_{\ell + i} -  x_{2 \ell -k -1 + i} x_{k+1} \mid \ell \in [2,k], 2\ell \leq k+1, k-2 \ell +2 < i \leq k - \ell \}$

\item $B_4 = \{ x_{\ell}^2 -  x_{2 \ell-k -1} x_{k+1} \mid \ell \in [2,k], 2\ell > k+1 \}$

\item $B_5 = \{ x_{\ell}x_{\ell + i} -  x_{2 \ell -k -1 + i} x_{k+1} \mid \ell \in [2,k], 2\ell > k+1, 1 \leq i \leq k - \ell \}$

\end{itemize}

\noindent Now, we proceed by Buchberger's criterion \ref{Buchberger's criterion}. The prrof is divided 
into various cases and subcases.\\

\noindent\textbf{Case 1.} If $f,g \in B_1$ and $f \neq g$ then $\mathrm{LT}(f) = x_{\ell}^2$ and $\mathrm{LT}(g) = x_{\ell'}^2$ for some $\ell, \ell' \in [2,k]$ and $\ell \neq \ell'$. We have $\mathrm{gcd}(\mathrm{LT}(f),\mathrm{LT}(g)) = 1,$ and hence $S(f,g) \rightarrow 0.$\\

\noindent\textbf{Case 2.} If $f \in B_1$ and $g \in B_2$ then $f = x_{\ell}^2 - x_1 x_{2 \ell -1}$ for some $\ell \in [2,k]$, $2 \ell \leq k+1$, and for a some  $  i \in  [1,k -2 \ell' + 2]$, $g = x_{\ell'}x_{\ell' + i} - x_1 x_{2 \ell' -1 + i}$ for some $\ell' \in [2,k], 2\ell' \leq k+1$.\\
\textbf{Subcase 2.1.} Let $\ell < \ell'$. We have $\mathrm{LT}(f) = x_{\ell}^2$ and $\mathrm{LT}(g) = x_{\ell'}x_{\ell'+i}$ for some $\ell, \ell' \in [2,k]$. Since $\ell < \ell'$, we have $\mathrm{gcd}(\mathrm{LT}(f),\mathrm{LT}(g)) = 1,$ and hence $S(f,g) \rightarrow 0.$\\
\textbf{Subcase 2.2.} Let $\ell = \ell'$. We have 
\begin{align*}
S(f,g) = x_{\ell+i}(x_{\ell}^2 - x_1 x_{2 \ell -1}) - x_{\ell}(x_{\ell}x_{\ell + i} - x_1 x_{2 \ell -1 + i}) = -x_1(x_{2\ell-1}x_{\ell+i} - x_{\ell}x_{2\ell -1 + i})
\end{align*}
If $3 \ell+ i -1 \leq k+2$, then 
\[
S(f,g) = -x_1(x_{2\ell-1}x_{\ell+i}-x_1 x_{3\ell + i -2})+ x_1(x_{\ell}x_{2\ell -1 + i} - x_1 x_{3\ell + i -2}).
\]
If $3 \ell+ i -1 > k+2$, then 
\[
S(f,g) = -x_1(x_{2\ell-1}x_{\ell+i}-x_{3 \ell + i -(k+2)}x_{k+1})+ x_1(x_{\ell}x_{2\ell -1 + i} -x_{3 \ell + i -(k+2)}x_{k+1}).
\]
Hence, we have $S(f,g) \rightarrow 0.$\\
\textbf{Subcase 2.3.} Let $\ell > \ell'$. If $\mathrm{gcd}(\mathrm{LT}(f),\mathrm{LT}(g)) = 1$ then $S(f,g) \rightarrow 0.$ Assume that $\mathrm{gcd}(\mathrm{LT}(f),\mathrm{LT}(g)) \neq 1$ then $\ell'+ i =\ell$. Therefore,
\[
S(f,g) = x_{\ell'}(x_{\ell}^2 - x_1 x_{2 \ell -1}) - x_{\ell}(x_{\ell'}x_{\ell' + i} - x_1 x_{2 \ell' -1 + i}) = -x_1(x_{\ell'}x_{2 \ell - 1} - x_{\ell}x_{2 \ell'+i-1}).
\]
If $2 \ell + \ell' - 1 \leq k+2$, then
\[
S(f,g) = -x_1(x_{\ell'}x_{2 \ell - 1}- x_1x_{2 \ell + \ell'-2}) + x_1(x_{\ell}x_{2 \ell'+i-1}-  x_1x_{2 \ell + \ell'-2}).
\]
If $2 \ell + \ell' - 1 > k+2$, then
\[
S(f,g) = -x_1(x_{\ell'}x_{2 \ell - 1}- x_{2\ell + \ell' - (k+2)}x_{k+1}) + x_1(x_{\ell}x_{2 \ell'+i-1}- x_{2\ell + \ell' - (k+2)}x_{k+1} ).
\]
Hence, we have $S(f,g) \rightarrow 0.$\\

\noindent\textbf{Case 3.} If $f \in B_1$ and $g \in B_3$ then  $f = x_{\ell}^2 - x_1 x_{2 \ell -1}$ for some $\ell \in [2,k]$, $2 \ell \leq k+1$, and for some  $  i \in  [k -2 \ell' + 3,k-\ell']$, $g = x_{\ell'}x_{\ell' + i} -  x_{2 \ell'- k-1 + i}x_{k+1}$ for some $\ell' \in [2,k], 2\ell' \leq k+1$.\\
\textbf{Subcase 3.1.} Let $\ell < \ell'$. We have $\mathrm{LT}(f) = x_{\ell}^2$ and $\mathrm{LT}(g) = x_{\ell'}x_{\ell'+i}$ for some $\ell, \ell' \in [2,k]$. Since $\ell < \ell'$, we have $\mathrm{gcd}(\mathrm{LT}(f),\mathrm{LT}(g)) = 1,$ and hence $S(f,g) \rightarrow 0.$\\
\textbf{Subcase 3.2.} Let $\ell = \ell'$. We have 
\[
S(f,g) = x_{\ell+i}(x_{\ell}^2 - x_1 x_{2 \ell -1}) - x_{\ell}(x_{\ell}x_{\ell + i} - x_{2 \ell -k -1 + i}x_{k+1}) = -x_1x_{2\ell-1}x_{\ell+i} + x_{2\ell -k -1 + i}x_{\ell}x_{k+1}
\]
Since $i > k-2\ell+2$, we have $3\ell + i -1 > \ell+k+1$. Also, $\ell \geq 2$, we have $3\ell + i -1 > k+2$. Again, since $i \leq k-\ell$, we have $3\ell-k-1+i \leq 2\ell -1$. Also, $2 \ell \leq k+2$, we have $3\ell-k-1+i \leq k+2$. Therefore, we have
\[
S(f,g) = -x_1(x_{2\ell-1}x_{\ell+i} - x_{3\ell +i -(k+2)}x_{k+1})+x_{k+1}(x_{2\ell - k -1 +i}x_{\ell} - x_1x_{3\ell +i -(k+2)})
\]
Hence, we have $S(f,g) \rightarrow 0.$\\
\textbf{Subcase 3.3.} Let $\ell > \ell'$. Suppose $\mathrm{gcd}(\mathrm{LT}(f),\mathrm{LT}(g)) \neq 1,$ then we have $\ell'+i = \ell.$ Since $i > k-2\ell+2$, we have $\ell  > k - \ell' + 2$. Now, since $2 \ell' \leq k+1$ and $\ell > \ell'$, we have $k-\ell' \geq \ell$. This impies $\ell \geq \ell+ 2$, which is a contradiction. Therefore, $\mathrm{gcd}(\mathrm{LT}(f),\mathrm{LT}(g)) = 1$ and hence $S(f,g) \rightarrow 0.$\\

\noindent\textbf{Case 4.} If $f \in B_1$ and $g \in B_4$ then $\mathrm{LT}(f) = x_{\ell}^2$ such that $2 \ell \leq k+1$ and $\mathrm{LT}(g) = x_{\ell'}^2$ such that $2 \ell' > k+1$. Therefore, we have $\mathrm{gcd}(\mathrm{LT}(f),\mathrm{LT}(g)) = 1,$ and hence $S(f,g) \rightarrow 0.$\\

\noindent\textbf{Case 5.} If $f \in B_1$ and $g \in B_5$ then $\mathrm{LT}(f) = x_{\ell}^2$ such that $2 \ell \leq k+1$ and $\mathrm{LT}(g) = x_{\ell'}x_{\ell'+i}$ such that $2 \ell' > k+1$. Therefore, we have $\mathrm{gcd}(\mathrm{LT}(f),\mathrm{LT}(g)) = 1,$ and hence $S(f,g) \rightarrow 0.$\\

\noindent\textbf{Case 6.} If $f,g \in B_2$ then $\mathrm{LT}(f) = x_{\ell}x_{\ell+i}$ such that $2 \ell \leq k+1, i \in [1,k-2\ell+2]$ and $\mathrm{LT}(g) = x_{\ell'}x_{\ell'+i'}$ such that $2 \ell' \leq k+1,i' \in [1,k-2\ell'+2]$.\\
\textbf{Subcase 6.1.} Let $\ell = \ell'$ and $i \neq i'$. We have
\[
S(f,g) = x_{\ell+i'}(x_{\ell}x_{\ell + i} - x_1 x_{2 \ell -1 + i})-x_{\ell + i}(x_{\ell}x_{\ell + i'} - x_1 x_{2 \ell -1 + i'}) =-x_1(x_{\ell+i'}x_{2\ell-1+i} - x_{\ell+i}x_{2\ell-1+i'})
\] 
If $3\ell + i + i' -1 \leq k+2$, then 
\[
S(f,g) = -x_1(x_{\ell+i'}x_{2\ell-1+i} - x_1x_{3 \ell + i + i' - 2}) + x_1(x_{\ell+i}x_{2\ell-1+i'} - x_1x_{3 \ell + i + i' - 2}).
\]
If $3\ell + i + i' -1 > k+2$, then 
\[
S(f,g) = -x_1(x_{\ell+i'}x_{2\ell-1+i} - x_{3 \ell + i + i' -(k +2)}x_{k+1}) + x_1(x_{\ell+i}x_{2\ell-1+i'} -x_{3 \ell + i + i' -(k +2)}x_{k+1}).
\]
Hence, we have $S(f,g) \rightarrow 0.$\\
\textbf{Subcase 6.2.} Let $\ell \neq \ell'$ and $\mathrm{gcd}(\mathrm{LT}(f),\mathrm{LT}(g)) \neq 1$. WLOG assume that $\ell < \ell'$. We have either $\ell + i = \ell'$ or $\ell + i = \ell' + i'$. Suppose $\ell + i = \ell'$, then 
\[
S(f,g) = x_{\ell'+i'}(x_{\ell}x_{\ell'} - x_1x_{\ell+\ell'-1})-x_{\ell}(x_{\ell'}x_{\ell' + i'} - x_1 x_{2 \ell' -1 + i'}) = -x_1(x_{\ell'+i'}x_{\ell+\ell'-1} - x_{\ell}x_{2 \ell'-1+i'}).
\]
If $2\ell' + \ell + i' -1 \leq k+2$, then 
\[
S(f,g) = -x_1(x_{\ell'+i'}x_{\ell+\ell'-1} - x_1x_{2\ell' + \ell + i' -2}) + x_1(x_{\ell}x_{2 \ell'-1+i'} -  x_1x_{2\ell' + \ell + i' -2}).
\]
If $2\ell' + \ell + i' -1 > k+2$, then 
\[
S(f,g) = -x_1(x_{\ell'+i'}x_{\ell+\ell'-1} - x_{2\ell' + \ell + i' -(k+2)}x_{k+1}) + x_1(x_{\ell}x_{2 \ell'-1+i'} - x_{2\ell' + \ell + i' -(k+2)}x_{k+1}).
\]
Now, suppose $\ell + i = \ell' + i'$, then 
\[
S(f,g) = x_{\ell'}(x_{\ell}x_{\ell+i} - x_1x_{2\ell-1 + i})-x_{\ell}(x_{\ell'}x_{\ell' + i'} - x_1 x_{2 \ell' -1 + i'}) = -x_1(x_{\ell'}x_{2\ell-1 + i} - x_{\ell}x_{2 \ell'-1+i'}).
\]
If $2\ell + \ell' -1 + i = 2 \ell' + \ell -1 + i' \leq k+2$, then
\[
S(f,g) = -x_1(x_{\ell'}x_{2\ell-1 + i} - x_1x_{2\ell + \ell' + i -2}) + x_1(x_{\ell}x_{2 \ell'-1+i'} -  x_1x_{2\ell' + \ell + i' -2}).
\]
If $2\ell + \ell' -1 + i = 2 \ell' + \ell -1 + i' > k+2$, then
\[
S(f,g) = -x_1(x_{\ell'}x_{2\ell-1 + i} - x_{2\ell + \ell' + i -(k+2)}x_{k+1}) + x_1(x_{\ell}x_{2 \ell'-1+i'} - x_{2\ell + \ell' + i -(k+2)}x_{k+1}).
\]
Hence, we have $S(f,g) \rightarrow 0.$\\

\noindent\textbf{Case 7.} If $f \in B_2$ and $g \in B_3$ then $\mathrm{LT}(f) = x_{\ell}x_{\ell+i}$ such that $2 \ell \leq k+1, i \in [1,k-2\ell+2]$ and $\mathrm{LT}(g) = x_{\ell'}x_{\ell'+i'}$ such that $2 \ell' \leq k+1,i' \in [k-2\ell' + 3,k-\ell']$.\\
\textbf{Subcase 7.1.} Let $\ell = \ell'$ and $\mathrm{gcd}(\mathrm{LT}(f),\mathrm{LT}(g)) \neq 1$. We have
\begin{align*}
S(f,g) &= x_{\ell + i'}(x_{\ell}x_{\ell + i} - x_1 x_{2 \ell -1 + i})-x_{\ell + i}( x_{\ell}x_{\ell + i'} -  x_{2 \ell- k-1 + i'}x_{k+1})\\
& = -x_1x_{\ell+i'}x_{2 \ell -1 + i} + x_{\ell + i}x_{2 \ell - k - 1 + i'}x_{k+1}.
\end{align*}
Since $3 \ell + i + i' - 1 \geq k + \ell + 4 > k+2,$ we have 
\[
S(f,g) = -x_1(x_{\ell+i'}x_{2 \ell -1 + i} - x_{3\ell+ i + i'-(k+2)}x_{k+1}) + x_{k+1}( x_{\ell + i}x_{2 \ell - k - 1 + i'}- x_{3\ell+ i + i'-(k+2)}x_{k+1}).
\]
Hence, we have $S(f,g) \rightarrow 0.$\\
\textbf{Subcase 7.2.} Let $\ell < \ell'$ and $\mathrm{gcd}(\mathrm{LT}(f),\mathrm{LT}(g)) \neq 1$. We have either $\ell + i = \ell'$ or $\ell + i = \ell' + i'$. Suppose $\ell + i = \ell'$, then
\begin{align*}
S(f,g) &=  x_{\ell'+i'}(x_{\ell}x_{\ell'} - x_1x_{\ell+\ell'-1})-x_{\ell}(x_{\ell'}x_{\ell' + i'}-x_{2\ell'-k-1+i'}x_{k+1})\\
&= -x_1x_{\ell'+i'}x_{\ell + \ell' -1} + x_{\ell}x_{2 \ell' - k -1 + i'}x_{k+1}. 
\end{align*}
Since $2 \ell' + \ell+ i' - 1 \geq k +\ell' - \ell + 2 > k+2,$ we have 
\[
S(f,g) = -x_1(x_{\ell'+i'}x_{\ell + \ell' -1} - x_{2\ell'+\ell+i'-(k+2)}x_{k+1})+ x_{k+1}( x_{\ell}x_{2 \ell' - k -1 + i'}- x_{2\ell'+\ell+i'-(k+2)}x_{k+1}).
\]
Suppose $\ell + i = \ell'+i'$, then
\begin{align*}
S(f,g) &=   x_{\ell'}(x_{\ell}x_{\ell'+i'}-x_1x_{\ell + \ell' -1 +i'}) - x_{\ell}(x_{\ell'}x_{\ell' + i'}-x_{2 \ell'-k-1 + i'}x_{k+1})\\
&= -x_1x_{\ell'}x_{\ell + \ell' -1 +i'} + x_{\ell}x_{2 \ell'-k-1 + i'}x_{k+1}.
\end{align*}
Since $i' > k-2 \ell' +2,$ we have $2 \ell' + l + i' -1 > k+ \ell - 1 \geq k+2.$ Again, since $i' \leq k - \ell'$, we have $2 \ell' + \ell + i' - k -1 \leq 2 \ell' -1 \leq k+1$. Therefore, we have 
\[
S(f,g) = -x_1(x_{\ell'}x_{\ell + \ell' -1 +i'} - x_{2 \ell' + \ell + i' -(k+2)}x_{k+1}) + x_{k+1}(x_{\ell}x_{2 \ell'-k-1 + i'} - x_1x_{2 \ell' + \ell + i' - k -2})
\]
Hence, we have $S(f,g) \rightarrow 0.$\\
\textbf{Subcase 7.3.} Let $\ell > \ell'$ and $\mathrm{gcd}(\mathrm{LT}(f),\mathrm{LT}(g)) \neq 1$. Therefore we have $\ell + i = \ell'+i'$. Since $ i \leq k-2 \ell + 2,$ we have $\ell + i \leq k - \ell + 2$. Now, since $\ell > \ell'$ and $\ell+ i = \ell'+i'$, we have $\ell'+ i' \leq k- \ell' + 2.$ But, since $i' > k - 2\ell'+2$, we have $\ell' + i' > k- \ell' + 2$. This is a contradiction and therefore $\mathrm{gcd}(\mathrm{LT}(f),\mathrm{LT}(g)) = 1$. Hence, we have $S(f,g) \rightarrow 0.$\\

\noindent\textbf{Case 8.} If $f \in B_2$ and $g \in B_4$ then $\mathrm{LT}(f) = x_{\ell}x_{\ell+i}$ such that $2 \ell \leq k+1, i \in [1,k-2\ell+2]$ and $\mathrm{LT}(g) = x_{\ell'}^2$ such that $2 \ell' > k+1$. Suppose $\mathrm{gcd}(\mathrm{LT}(f),\mathrm{LT}(g)) \neq 1$, we have $\ell+ i = \ell'$. Therefore,
\[
S(f,g) = x_{\ell'}(x_{\ell}x_{\ell'} - x_1x_{\ell + \ell'-1}) - x_{\ell}(x_{\ell'}^2-x_{2 \ell'-k-1}x_{k+1}) = -x_1x_{\ell'}x_{\ell + \ell'-1}+x_{\ell}x_{2 \ell' - k -1}x_{k+1}.
\]
Since, $2 \ell' > k+1,$ we have $2 \ell' + \ell -1 > k + \ell \geq k + 2.$ Also, since $\ell + \ell' - 1 \leq k+1,$ we have $2\ell' + \ell - k -1 \leq \ell' + 1 \leq k+1.$ Therefore, we have 
\[
S(f,g) = -x_1(x_{\ell'}x_{\ell + \ell'-1}- x_{2 \ell' + \ell -(k+2)}x_{k+1}) + x_{k+1}(x_{\ell}x_{2 \ell' - k -1} - x_1x_{2\ell' + \ell -k -2}).
\]
Hence, we have $S(f,g) \rightarrow 0.$\\

\noindent\textbf{Case 9.} If $f \in B_2$ and $g \in B_5$ then $\mathrm{LT}(f) = x_{\ell}x_{\ell+i}$ such that $2 \ell \leq k+1, i \in [1,k-2\ell+2]$ and $\mathrm{LT}(g) = x_{\ell'}x_{\ell'+i'}$ such that $2 \ell' > k+1,i' \in [1,k-\ell']$. Suppose $\mathrm{gcd}(\mathrm{LT}(f),\mathrm{LT}(g)) \neq 1$, we have either $\ell+ i = \ell'$ or $\ell+ i = \ell'+i'$. For $\ell+ i = \ell'$, we have 
\begin{align*}
S(f,g) &= x_{\ell' + i'}(x_{\ell}x_{\ell'} - x_1x_{\ell + \ell'-1}) - x_{\ell}(x_{\ell'}x_{\ell'+i'}-x_{2 \ell'-k-1+i'}x_{k+1})\\
&= -x_1x_{\ell' + i'}x_{\ell + \ell'-1} + x_{\ell}x_{2 \ell'-k-1+i'}x_{k+1}.
\end{align*}
Since, $2 \ell' > k+1,$ we have $2 \ell' + \ell + i' -1 > k + \ell \geq k + 2.$ Also, since $\ell + \ell' - 1 \leq k+1$ and $\ell'+i' \leq k+1$ we have $2\ell' + \ell - k -1 + i' \leq k+2.$ Therefore, we have 
\[
S(f,g) = -x_1(x_{\ell'+i'}x_{\ell + \ell'-1}- x_{2 \ell' + \ell +i' -(k+2)}x_{k+1}) + x_{k+1}(x_{\ell}x_{2 \ell' - k -1+i'} - x_1x_{2\ell' + \ell +i' -k -2}).
\]
For $\ell+ i = \ell'+i'$, we have
\begin{align*}
S(f,g) &= x_{\ell'}(x_{\ell}x_{\ell+i} - x_1x_{2\ell -1 + i}) - x_{\ell}(x_{\ell'}x_{\ell'+i'}-x_{2 \ell'-k-1+i'}x_{k+1})\\
&= -x_1x_{\ell'}x_{2\ell-1 + i} + x_{\ell}x_{2 \ell'-k-1+i'}x_{k+1}\\
&= -x_1x_{\ell'}x_{\ell+ \ell'-1 + i'} + x_{\ell}x_{2 \ell'-k-1+i'}x_{k+1}
\end{align*}
Since, $2 \ell' > k+1,$ we have $2 \ell' + \ell + i' -1 > k + \ell \geq k + 2.$ Also, since $\ell + \ell' - 1 \leq k+1$ and $\ell'+i' \leq k+1$ we have $2\ell' + \ell - k -1 + i' \leq k+2.$ Therefore, we have 
\[
S(f,g) = -x_1(x_{\ell'}x_{\ell + \ell'-1 + i'}- x_{2 \ell' + \ell +i' -(k+2)}x_{k+1}) + x_{k+1}(x_{\ell}x_{2 \ell' - k -1+i'} - x_1x_{2\ell' + \ell +i' -k -2}).
\]
Hence, we have $S(f,g) \rightarrow 0.$\\

\noindent\textbf{Case 10.} If $f,g \in B_3$ then $\mathrm{LT}(f) = x_{\ell}x_{\ell+i}$ such that $2 \ell \leq k+1, i \in [k-2\ell+3,k-\ell]$ and $\mathrm{LT}(g) = x_{\ell'}x_{\ell'+i'}$ such that $2 \ell' \leq k+1,i' \in [k-2\ell'+3,k-\ell]$.\\
\textbf{Subcase 10.1.} Let $\ell = \ell'$ and $i \neq i'$. We have
\begin{align*}
S(f,g) &= x_{\ell+i'}(x_{\ell}x_{\ell + i} -  x_{2 \ell -1 - k + i}x_{k+1})-x_{\ell + i}(x_{\ell}x_{\ell + i'} - x_{2 \ell -1 -k + i'}x_{k+1})\\
&=-x_{k+1}(x_{\ell+i'}x_{2\ell-1 - k +i} - x_{\ell+i}x_{2\ell-1 - k +i'})
\end{align*}
 
If $3\ell + i + i' -k  -1 \leq k+2$, then 
\[
S(f,g) = -x_{k+1}(x_{\ell+i'}x_{2\ell-1 - k +i} - x_1x_{3 \ell + i + i' - k - 2}) + x_{k+1}(x_{\ell+i}x_{2\ell-1 - k +i'} - x_1x_{3 \ell + i + i' -k - 2}).
\]
If $3\ell + i + i' -k  -1 > k+2$, then 
\[
S(f,g) = -x_{k+1}(x_{\ell+i'}x_{2\ell-1+i} - x_{3 \ell + i + i' -2k -2}x_{k+1}) + x_{k+1}(x_{\ell+i}x_{2\ell-1+i'} -x_{3 \ell + i + i'-2k -2}x_{k+1}).
\]
Hence, we have $S(f,g) \rightarrow 0.$\\
\textbf{Subcase 10.2.} Let $\ell \neq \ell'$ and $\mathrm{gcd}(\mathrm{LT}(f),\mathrm{LT}(g)) \neq 1$. WLOG assume that $\ell < \ell'$. We have either $\ell + i = \ell'$ or $\ell + i = \ell' + i'$. Suppose $\ell + i = \ell'$, then 
\begin{align*}
S(f,g) &= x_{\ell'+i'}(x_{\ell}x_{\ell'} - x_{\ell+\ell'-k -1}x_{k+1})-x_{\ell}(x_{\ell'}x_{\ell' + i'} -  x_{2 \ell' -1 - k + i'}x_{k+1})\\
& = -x_{k+1}(x_{\ell'+i'}x_{\ell+\ell'-k -1} - x_{\ell}x_{2 \ell'-1 - k +i'}).
\end{align*}

If $2\ell' + \ell + i' -k -1 \leq k+2$, then 
\[
S(f,g) = -x_{k+1}(x_{\ell'+i'}x_{\ell+\ell'-k -1} - x_1x_{2\ell' + \ell + i'-k -2}) + x_{k+1}(x_{\ell}x_{2 \ell'-1-k+i'} -  x_1x_{2\ell' + \ell + i'-k -2}).
\]
If $2\ell' + \ell + i' -k-1 > k+2$, then 
\[
S(f,g) = -x_{k+1}(x_{\ell'+i'}x_{\ell+\ell'-k-1} - x_{2\ell' + \ell + i'-2k -2}x_{k+1}) + x_{k+1}(x_{\ell}x_{2 \ell'-1-k+i'} - x_{2\ell' + \ell + i'-2k-2}x_{k+1}).
\]
Now, suppose $\ell + i = \ell' + i'$, then 
\begin{align*}
S(f,g) &= x_{\ell'}(x_{\ell}x_{\ell+i} - x_{2\ell-1 -k+ i}x_{k+1})-x_{\ell}(x_{\ell'}x_{\ell' + i'} - x_{2 \ell' -1 -k + i'}x_{k+1})\\
& = -x_{k+1}(x_{\ell'}x_{2\ell-1 -k + i} - x_{\ell}x_{2 \ell'-1 - k +i'}).
\end{align*}

If $2\ell + \ell' -1 -k + i = 2 \ell' + \ell -1 -k + i' \leq k+2$, then
\[
S(f,g) = -x_{k+1}(x_{\ell'}x_{2\ell-1-k + i} - x_1x_{2\ell + \ell' + i -k -2}) + x_{k+1}(x_{\ell}x_{2 \ell'-1-k +i'} -  x_1x_{2\ell' + \ell + i'-k -2}).
\]
If $2\ell + \ell' -1 -k + i = 2 \ell' + \ell -1 -k + i' > k+2$, then
\[
S(f,g) = -x_{k+1}(x_{\ell'}x_{2\ell-1-k + i} - x_{2\ell + \ell' + i -2k -2}x_{k+1}) + x_{k+1}(x_{\ell}x_{2 \ell'-1-k+i'} - x_{2\ell + \ell' + i-2k -2}x_{k+1}).
\]
Hence, we have $S(f,g) \rightarrow 0.$\\

\noindent\textbf{Case 11.} If $f \in B_3$ and $g \in B_4$ then $\mathrm{LT}(f) = x_{\ell}x_{\ell+i}$ such that $2 \ell \leq k+1, i \in [k-2\ell+3,k-\ell]$ and $\mathrm{LT}(g) = x_{\ell'}^2$ such that $2 \ell' > k+1$. Suppose $\mathrm{gcd}(\mathrm{LT}(f),\mathrm{LT}(g)) \neq 1$, we have $\ell+ i = \ell'$. Therefore,
\[
S(f,g) = x_{\ell'}(x_{\ell}x_{\ell'} - x_{\ell + \ell'-k-1}x_{k+1}) - x_{\ell}(x_{\ell'}^2-x_{2 \ell'-k-1}x_{k+1}) = -x_{k+1}x_{\ell'}x_{\ell + \ell' - k -1}+x_{\ell}x_{2 \ell' - k -1}x_{k+1}.
\]
If $2 \ell' + \ell - k -1 \leq k+2,$ then
\[
S(f,g) = -x_{k+1}(x_{\ell'}x_{\ell + \ell' - k -1} - x_1x_{2 \ell' + \ell - k -2}) + x_{k+1}(x_{\ell}x_{2 \ell' - k -1}- x_1x_{2 \ell' + \ell - k -2}).
\]
If $2 \ell' + \ell - k -1 > k+2,$ then
\[
S(f,g) = -x_{k+1}(x_{\ell'}x_{\ell + \ell' - k -1} - x_{2 \ell' + \ell - k -(k+2)}x_{k+1}) + x_{k+1}(x_{\ell}x_{2 \ell' - k -1}- x_{2 \ell' + \ell - k -(k+2)}x_{k+1}).
\]
Hence, we have $S(f,g) \rightarrow 0.$\\

\noindent\textbf{Case 12.} If $f \in B_3$ and $g \in B_5$ then $\mathrm{LT}(f) = x_{\ell}x_{\ell+i}$ such that $2 \ell \leq k+1, i \in [k-2\ell+3,k-\ell]$ and $\mathrm{LT}(g) = x_{\ell'}x_{\ell'+i'}$ such that $2 \ell' > k+1,i' \in [1,k-\ell']$. Suppose $\mathrm{gcd}(\mathrm{LT}(f),\mathrm{LT}(g)) \neq 1$, we have either $\ell+ i = \ell'$ or $\ell+ i = \ell'+i'$. For $\ell+ i = \ell'$, we have 
\begin{align*}
S(f,g) &= x_{\ell' + i'}(x_{\ell}x_{\ell'} - x_{\ell + \ell'-k-1}x_{k+1}) - x_{\ell}(x_{\ell'}x_{\ell'+i'}-x_{2 \ell'-k-1+i'}x_{k+1})\\
&= -x_{k+1}x_{\ell' + i'}x_{\ell + \ell'-k-1} + x_{\ell}x_{2 \ell'-k-1+i'}x_{k+1}.
\end{align*}
If $2 \ell' + \ell + i'- k -1 \leq k+2,$ then
\[
S(f,g) = -x_{k+1}(x_{\ell'+i'}x_{\ell + \ell' - k -1} - x_1x_{2 \ell' + \ell + i' - k -2}) + x_{k+1}(x_{\ell}x_{2 \ell' - k -1 + i'}- x_1x_{2 \ell' + \ell + i'- k -2}).
\]
If $2 \ell' + \ell + i'- k -1 > k+2,$ then
\[
S(f,g) = -x_{k+1}(x_{\ell'+i'}x_{\ell + \ell' - k -1} - x_{2 \ell' + \ell +i' - k -(k+2)}x_{k+1}) + x_{k+1}(x_{\ell}x_{2 \ell' - k -1}- x_{2 \ell' + \ell + i'- k -(k+2)}x_{k+1}).
\]
Now, suppose $\ell + i = \ell' + i'$, then 
\begin{align*}
S(f,g) &= x_{\ell'}(x_{\ell}x_{\ell+i} - x_{2\ell-1 -k+ i}x_{k+1})-x_{\ell}(x_{\ell'}x_{\ell' + i'} - x_{2 \ell' -1 -k + i'}x_{k+1})\\
& = -x_{k+1}(x_{\ell'}x_{2\ell-1 -k + i} - x_{\ell}x_{2 \ell'-1 - k +i'}).
\end{align*}

If $2\ell + \ell' -1 -k + i = 2 \ell' + \ell -1 -k + i' \leq k+2$, then
\[
S(f,g) = -x_{k+1}(x_{\ell'}x_{2\ell-1-k + i} - x_1x_{2\ell + \ell' + i -k -2}) + x_{k+1}(x_{\ell}x_{2 \ell'-1-k +i'} -  x_1x_{2\ell' + \ell + i'-k -2}).
\]
If $2\ell + \ell' -1 -k + i = 2 \ell' + \ell -1 -k + i' > k+2$, then
\[
S(f,g) = -x_{k+1}(x_{\ell'}x_{2\ell-1-k + i} - x_{2\ell + \ell' + i -k -(k+2)}x_{k+1}) + x_{k+1}(x_{\ell}x_{2 \ell'-1-k+i'} - x_{2\ell + \ell' + i-k -(k+2)}x_{k+1}).
\]
Hence, we have $S(f,g) \rightarrow 0.$\\

\noindent\textbf{Case 13.} If $f,g \in B_4$ and $f \neq g$ then $\mathrm{LT}(f) = x_{\ell}^2$ and $\mathrm{LT}(g) = x_{\ell'}^2$ for some $\ell, \ell' \in [2,k]$ and $\ell \neq \ell'$. We have $\mathrm{gcd}(\mathrm{LT}(f),\mathrm{LT}(g)) = 1,$ and hence $S(f,g) \rightarrow 0.$\\

\noindent\textbf{Case 14.} If $f \in B_4$ and $g \in B_5$ then $\mathrm{LT}(f) = x_{\ell}^2$ such that $2 \ell > k+1$ and $\mathrm{LT}(g) = x_{\ell'}x_{\ell'+i}$ such that $2 \ell' > k+1, i \in [1,k-\ell]$. \\
\textbf{Subcase 14.1.} Let $\ell < \ell'$. We have $\mathrm{LT}(f) = x_{\ell}^2$ and $\mathrm{LT}(g) = x_{\ell'}x_{\ell'+i}$ for some $\ell, \ell' \in [2,k]$. Since $\ell < \ell'$, we have $\mathrm{gcd}(\mathrm{LT}(f),\mathrm{LT}(g)) = 1,$ and hence $S(f,g) \rightarrow 0.$\\
\textbf{Subcase 14.2.} Let $\ell = \ell'$. We have 
\begin{align*}
S(f,g) &= x_{\ell+i}(x_{\ell}^2 -  x_{2 \ell-k -1}x_{k+1}) - x_{\ell}(x_{\ell}x_{\ell + i} -  x_{2 \ell -k -1 + i}x_{k+1})\\
& = -x_{k+1}(x_{2\ell-k-1}x_{\ell+i} - x_{\ell}x_{2\ell -1 -k + i}).
\end{align*}
If $3 \ell+ i -k -1 \leq k+2$, then 
\[
S(f,g) = -x_{k+1}(x_{2\ell-k-1}x_{\ell+i}-x_1 x_{3\ell + i-k -2})+ x_{k+1}(x_{\ell}x_{2\ell -1 -k + i} - x_1 x_{3\ell + i-k -2}).
\]
If $3 \ell+ i -1 > k+2$, then 
\[
S(f,g) = -x_{k+1}(x_{2\ell-k-1}x_{\ell+i}-x_{3 \ell + i -k -(k+2)}x_{k+1})+ x_{k+1}(x_{\ell}x_{2\ell -1 -k + i} -x_{3 \ell + i-k -(k+2)}x_{k+1}).
\]
Hence, we have $S(f,g) \rightarrow 0.$\\
\textbf{Subcase 14.3.} Let $\ell > \ell'$. If $\mathrm{gcd}(\mathrm{LT}(f),\mathrm{LT}(g)) = 1$ then $S(f,g) \rightarrow 0.$ WLOG assume that $\mathrm{gcd}(\mathrm{LT}(f),\mathrm{LT}(g)) \neq 1$ then $\ell'+ i =\ell$. Therefore,
\[
S(f,g) = x_{\ell'}(x_{\ell}^2 -  x_{2 \ell-k -1}x_{k+1}) - x_{\ell}(x_{\ell'}x_{\ell' + i} - x_1 x_{2 \ell' -k -1 + i}) = -x_{k+1}(x_{\ell'}x_{2 \ell-k - 1} - x_{\ell}x_{2 \ell'+i-k -1}).
\]
If $2 \ell + \ell' -k - 1 \leq k+2$, then
\[
S(f,g) = -x_{k+1}(x_{\ell'}x_{2 \ell-k - 1}- x_1x_{2 \ell + \ell'-k-2}) + x_{k+1}(x_{\ell}x_{2 \ell'+i-k-1}-  x_1x_{2 \ell + \ell'-k-2}).
\]
If $2 \ell + \ell'-k - 1 > k+2$, then
\[
S(f,g) = -x_{k+1}(x_{\ell'}x_{2 \ell-k - 1}- x_{2\ell + \ell' -k- (k+2)}x_{k+1}) + x_{k+1}(x_{\ell}x_{2 \ell'+i-k-1}- x_{2\ell + \ell'-k - (k+2)}x_{k+1} ).
\]
Hence, we have $S(f,g) \rightarrow 0.$\\

\noindent\textbf{Case 15.} If $f,g \in B_5$ then $\mathrm{LT}(f) = x_{\ell}x_{\ell+i}$ such that $2 \ell > k+1, i \in [1,k-\ell]$ and $\mathrm{LT}(g) = x_{\ell'}x_{\ell'+i'}$ such that $2 \ell' > k+1,i' \in [1,k-\ell]$.\\
\textbf{Subcase 15.1.} Let $\ell = \ell'$ and $i \neq i'$. We have
\begin{align*}
S(f,g) &= x_{\ell+i'}(x_{\ell}x_{\ell + i} -  x_{2 \ell -1 - k + i}x_{k+1})-x_{\ell + i}(x_{\ell}x_{\ell + i'} - x_{2 \ell -1 -k + i'}x_{k+1})\\
&=-x_{k+1}(x_{\ell+i'}x_{2\ell-1 - k +i} - x_{\ell+i}x_{2\ell-1 - k +i'})
\end{align*}
 
If $3\ell + i + i' -k  -1 \leq k+2$, then 
\[
S(f,g) = -x_{k+1}(x_{\ell+i'}x_{2\ell-1 - k +i} - x_1x_{3 \ell + i + i' - k - 2}) + x_{k+1}(x_{\ell+i}x_{2\ell-1 - k +i'} - x_1x_{3 \ell + i + i' -k - 2}).
\]
If $3\ell + i + i' -k  -1 > k+2$, then 
\[
S(f,g) = -x_{k+1}(x_{\ell+i'}x_{2\ell-1+i} - x_{3 \ell + i + i' -2k -2}x_{k+1}) + x_{k+1}(x_{\ell+i}x_{2\ell-1+i'} -x_{3 \ell + i + i'-2k -2}x_{k+1}).
\]
Hence, we have $S(f,g) \rightarrow 0.$\\
\textbf{Subcase 15.2.} Let $\ell \neq \ell'$ and $\mathrm{gcd}(\mathrm{LT}(f),\mathrm{LT}(g)) \neq 1$. WLOG assume that $\ell < \ell'$. We have either $\ell + i = \ell'$ or $\ell + i = \ell' + i'$. Suppose $\ell + i = \ell'$, then 
\begin{align*}
S(f,g) &= x_{\ell'+i'}(x_{\ell}x_{\ell'} - x_{\ell+\ell'-k -1}x_{k+1})-x_{\ell}(x_{\ell'}x_{\ell' + i'} -  x_{2 \ell' -1 - k + i'}x_{k+1})\\
& = -x_{k+1}(x_{\ell'+i'}x_{\ell+\ell'-k -1} - x_{\ell}x_{2 \ell'-1 - k +i'}).
\end{align*}

If $2\ell' + \ell + i' -k -1 \leq k+2$, then 
\[
S(f,g) = -x_{k+1}(x_{\ell'+i'}x_{\ell+\ell'-k -1} - x_1x_{2\ell' + \ell + i'-k -2}) + x_{k+1}(x_{\ell}x_{2 \ell'-1-k+i'} -  x_1x_{2\ell' + \ell + i'-k -2}).
\]
If $2\ell' + \ell + i' -k-1 > k+2$, then 
\[
S(f,g) = -x_{k+1}(x_{\ell'+i'}x_{\ell+\ell'-k-1} - x_{2\ell' + \ell + i'-2k -2}x_{k+1}) + x_{k+1}(x_{\ell}x_{2 \ell'-1-k+i'} - x_{2\ell' + \ell + i'-2k-2}x_{k+1}).
\]
Now, suppose $\ell + i = \ell' + i'$, then 
\begin{align*}
S(f,g) &= x_{\ell'}(x_{\ell}x_{\ell+i} - x_{2\ell-1 -k+ i}x_{k+1})-x_{\ell}(x_{\ell'}x_{\ell' + i'} - x_{2 \ell' -1 -k + i'}x_{k+1})\\
& = -x_{k+1}(x_{\ell'}x_{2\ell-1 -k + i} - x_{\ell}x_{2 \ell'-1 - k +i'}).
\end{align*}

If $2\ell + \ell' -1 -k + i = 2 \ell' + \ell -1 -k + i' \leq k+2$, then
\[
S(f,g) = -x_{k+1}(x_{\ell'}x_{2\ell-1-k + i} - x_1x_{2\ell + \ell' + i -k -2}) + x_{k+1}(x_{\ell}x_{2 \ell'-1-k +i'} -  x_1x_{2\ell' + \ell + i'-k -2}).
\]
If $2\ell + \ell' -1 -k + i = 2 \ell' + \ell -1 -k + i' > k+2$, then
\[
S(f,g) = -x_{k+1}(x_{\ell'}x_{2\ell-1-k + i} - x_{2\ell + \ell' + i -2k -2}x_{k+1}) + x_{k+1}(x_{\ell}x_{2 \ell'-1-k+i'} - x_{2\ell + \ell' + i-2k -2}x_{k+1}).
\]
Hence, we have $S(f,g) \rightarrow 0.$ This completes the proof.
\end{proof}

\section{Syzygies of $k[S_{a,d,k}]$}
Let $S = \langle a_1,\ldots,a_n \rangle \subset \mathbb{N}^r$ be an affine semigroup. Then the map 
$ \delta_0 : R=k[x_1,\ldots ,x_n] \longrightarrow k[S]:=\oplus_{a \in S} k {\bf t}^a$ such that $x_i \mapsto {\bf t}^{a_i} $ is $S$-graded surjective $k$-algebra homomorphism and the ideal $I_S = \mathrm{ker}(\delta_0)$ is homogeneous with respect to this grading. Now by using the $S$-graded Nakayama's lemma \cite[Proposition, 1.4]{briales-campillonakayama}, we can construct the graded $k$-algebra homomorphism $\delta_{i+1}: R^{\beta_{i+1}} \longrightarrow  R^{\beta_{i}} $ corresponding to a minimal set of homogeneous  generators of $\mathrm{ker}(\delta_i), i \geq 0.$ Hence, we obtain a minimal graded free resolution of $k[S]$ as $R$-module:
\[
\cdots R^{\beta_{i+1}} \xrightarrow{\delta_{i+1}}  R^{\beta_{i}} \xrightarrow{\delta_i} \cdots \rightarrow R^{\beta_1} \xrightarrow{\delta_1} R \xrightarrow{\delta_0} k[S] \rightarrow 0.
\]
$\beta_i$ is called the $i^{\mathrm{th}}$ Betti number of $k[S].$ For $i \geq 1,$ we have $\beta_i = \sum_{a \in S} \mathrm{dim}_k \frac{(\mathrm{ker}(\delta_{i-1}))_a}{({\bf \mathfrak{m}}\mathrm{ker}(\delta_{i-1}))_a}$, where $\mathfrak{m} = \langle x_1,\ldots,x_n \rangle$ is the homogeneous maximal ideal of $R$.\\ 
$\beta_{i,a}:= \mathrm{dim}_k \frac{(\mathrm{ker}(\delta_{i-1}))_a}{({\bf \mathfrak{m}}\mathrm{ker}(\delta_{i-1}))_a}$ is the number of generators of degree $a$ in minimal generating set of $\mathrm{ker}(\delta_{i-1}),$ called the $i^{\mathrm{th}}$ multigraded Betti number of $k[S]$ in degree $a$.\\
The Hilbert series of affine semigroup algebra $k[S]$ is defined as the formal sum of all monomials ${\bf t}^s = t_1^{s_1}\cdots t_r^{s_r},$ where $s \in S$. The Hilbert series of $k[S]$ is a rational function  of the form
\[
H(k[S],{\bf t}) = \sum_{s \in S} {\bf t}^s = \frac{\mathcal{K}(t_1,\ldots,t_r)}{\prod_{i=1}^n (1- {\bf t}^{a_i})},
\]
where $\mathcal{K}(t_1,\ldots,t_r)$ is a polynomial in $\mathbb{Z}[t_1,\ldots,t_r].$

In this section, we compute the graded minimal free resolution
and the Hilbert series of $k[S_{a,d,k}]$ for $k = 2,3,4.$ For $k = 2$, note that from the Theorem \ref{generatingset}, the defining ideal $I_{S_{a,d,2}}$ is generated by the single binomial $x_2^2-x_1x_3$. Therefore the minimal graded free resolution of $k[S_{a,d,2}]$ is given by
\[
 0 \longrightarrow R(-2(a+d)) \longrightarrow R \longrightarrow R/I_{S_{a,d,2}} \longrightarrow 0,
\]   
where $R = k[x_1,x_2,x_3].$ Hence by \cite[Proposition, 8.23]{millersturmfels}, the Hilbert series of $k[S_{a,d,2}]$ is given by
\[
H(k[S_{a,d,2}],{\bf t}) = \frac{1- {\bf t}^{2a+2d}}{\prod_{i=0}^2 (1-{\bf t}^{a+id})}.
\]
\begin{theorem}[{\bf (Buchsbaum-Eisenbud acyclicity criterion)}, {\cite[Theorem, 1.4.13]{herzogcmrings}}]  \label{Buchsbaum-Eisenbud acyclicity criterion}
	Let $R$ be a Noetherian ring and 
	\[
	F_.:0 \rightarrow F_s \xrightarrow{\delta_s} F_{s-1} \rightarrow \cdots \rightarrow F_1 \xrightarrow{\delta_1} F_0 \rightarrow 0
	\]
	a complex of finite free $R$-modules. Set $r_i = \sum_{j=i}^s(-1)^{j-i}$ rank $ F_j$. Then the following are equivalent:
	\begin{itemize}
		\item[(a)] $F_.$ is acyclic.
		\item[(b)] grade $I_{r_i}(\delta_i) \geq i$ for $i=1,\ldots,s$,
	\end{itemize}
	where $I_{r_i}(\delta_i)$ is the ideal generated by $r_i \times r_i$ minors of $\delta_i$.
\end{theorem}

\begin{proposition}\label{resk=3}
Suppose $k = 3$. Then the complex 
\[
F_.:  0 \rightarrow R^2 \xrightarrow{\delta_2} R^3 \xrightarrow{\delta_1} R \rightarrow R/I_{S_{a,d,3}} \rightarrow 0
\]
is a minimal graded free resolution of $k[S_{a,d,3}]$, where $R = k[x_1,x_2,x_3,x_4]$ and the maps $\delta_i$'s are given by 
\[
\delta_1 = 
\begin{pmatrix}
x_2^2-x_1x_3 & x_2x_3-x_1x_4 & x_3^2-x_2x_4
\end{pmatrix}
\]
and
\[
\delta_2 = 
\begin{pmatrix}
-x_3 & x_4 \\
x_2 &  -x_3 \\
-x_1 & x_2
\end{pmatrix} .
\]
\end{proposition}

\begin{proof}
From the Theorem \ref{generatingset}, we see that $\delta_1$ forms a minimal generating set for $I_{S_{a,d,3}}$. Note that $\delta_1 \delta_2 = 0.$ Therefore, $F_.$ is a chain complex. Let $r_i$ be the number defined in the Theorem \ref{Buchsbaum-Eisenbud acyclicity criterion}, we have $r_1 = 1, r_2 = 2.$ It is clear that $\mathrm{grade}(I_{r_1}(\delta_1)) \geq 1$. Observe that $\{x_3^2 - x_2x_4, x_2^2 - x_1x_3\} \subset I_{r_2}(\delta_2)$. Now, with respect to the reverse lexicographic order on $R$, the leading terms of these polynomials are mutually coprime. Hence $\mathrm{grade}(I_{r_2}(\delta_2)) \geq 2$. By the Theorem \ref{Buchsbaum-Eisenbud acyclicity criterion}, we conclude that $F_.$ is a graded free resolution of $k[S_{a,d,3}]$. Minimality of the resolution follows from the fact that all entries of $\delta_i's$ belong to the homogeneous maximal ideal $\textbf{m} = \langle x_1,x_2,x_3,x_4 \rangle$.
\end{proof}

\begin{proposition}\label{hilbertseriesk=3}
The Hilbert series of $k[S_{a,d,3}]$ is given by
\[
H(k[S_{a,d,3}];{\bf{t}}) = \frac{1- \sum_{i=2}^4 {\bf{t}}^{2a+id} + \sum_{i=4}^5{\bf{t}}^{3a+id}}{\prod_{i=0}^{3} (1- {\bf{t}}^{a+id})} .
\]
\end{proposition}
\begin{proof}
Let $\beta_{i,s}(k[S_{a,d,3}])$ denotes the $i^{\mathrm{th}}$ multigraded Betti number of $k[S_{a,k,3}]$ in degree $s$. Define
\[
C_i := \{ s \in S_{a,d,3} \mid \beta_{i,s}(k[S_{a,k,3}]) \neq 0 \} 
\]
Since $\mathrm{pdim}_R k[S_{a,d,3}] = 2,$ we have $C_i = \emptyset$ for $i > 2.$ Now by computing the degrees of syzygies from the Proposition \ref{resk=3}, we have
\[ 
C_0:= \{{\bf{0}}\}, \quad C_1:= \{2a+id \mid i = 2,3,4\}, \quad C_2:= \{3a+id \mid i = 4,5\}
\]
Using \cite[Proposition, 8.23]{millersturmfels}, we can write
\[
H(k[S_{a,d,3}];{\bf{t}})= \frac{\sum_{i = 0,s \in S_{a,d,3}}^2 (-1)^i \beta_{i,s}(k[S_{a,d,3}]){\bf{t}}^s}{\prod_{i = 0}^3 (1-{\bf{t}}^{a+id})}
\]
Note that, for any $s \in C_i$, $i =0,1,2$ we have  $\beta_{i,s}(k[S_{a,k,3}]) = 1.$ Hence, we have
\[
H(k[S_{a,d,3}];{\bf{t}}) = \frac{1- \sum_{i=2}^4 {\bf{t}}^{2a+id} + \sum_{i=4}^5{\bf{t}}^{3a+id}}{\prod_{i=0}^{3} (1- {\bf{t}}^{a+id})} .
\]

\end{proof}

\begin{proposition}\label{resk=4}
Suppose $k = 4$. Then the complex 
\[
F_.:  0 \rightarrow R^3 \xrightarrow{\delta_3} R^8 \xrightarrow{\delta_2} R^6 \xrightarrow{\delta_1} R \rightarrow R/I_{S_{a,d,4}} \rightarrow 0
\]
is a minimal graded free resolution of $k[S_{a,d,4}]$, where $R = k[x_1,x_2,x_3,x_4,x_5]$ and the maps $\delta_i$'s are given by
\[
\delta_1 = 
\begin{pmatrix}
x_2^2-x_1x_3 & x_2x_3-x_1x_4 & x_3^2-x_1x_5 & x_2x_4-x_1x_5 & x_3x_4-x_2x_5 & x_4^2-x_3x_5
\end{pmatrix}
\]

\[
\delta_2 = 
\begin{pmatrix}
-x_3 & 0 & -x_4 & 0 & x_5 & 0 & 0 & 0 \\
x_2 &  -x_3 & 0 & -x_4 & 0 & x_5 & x_5 & 0 \\
-x_1 & x_2 & 0 & 0 & 0 & -x_4 & 0 & x_5 \\
x_1 & 0 & x_2 & x_3 & -x_3 & 0 & -x_4  & x_5 \\
0 & -x_1 & -x_1 & 0 & x_2 & x_3 & 0 & -x_4 \\
0 & 0 & 0 & -x_1 & 0 & 0 & x_2 & x_3
\end{pmatrix}
\]
and 
\[
\delta_3 = 
\begin{pmatrix}
x_4 & -x_5 & 0 \\
0 & x_4 & -x_5 \\
-x_3 & 0 & x_5 \\
x_2 & -x_3 & 0 \\
0 & -x_3 & x_4 \\
-x_1 & x_2 & 0 \\
x_1 & 0 & -x_3 \\
0 & -x_1 & x_2
\end{pmatrix}.
\]
\end{proposition}

\begin{proof}
From the Theorem \ref{generatingset}, we see that $\delta_1$ forms a minimal generating set for $I_{S_{a,d,4}}$. Note that $\delta_i \delta_{i+1} = 0$ for $i = 1,2.$ Therefore, $F_.$ is a chain complex. Let $r_i$ be the number defined in the Theorem \ref{Buchsbaum-Eisenbud acyclicity criterion}, we have $r_1 = 1, r_2 = 5$ and $r_3 = 3.$ It is clear that $\mathrm{grade}(I_{r_1}(\delta_1)) \geq 1$. Let $R_i$ and $C_i$ denote the rows and colums of $\delta_j$'s. Consider the minors 
\begin{center}
$D_1:= |R_2~R_3~R_4~R_5~R_6 ~|~C_1~C_2~C_3~C_4~C_5|= x_1(x_2^2-x_1x_3)^2$\\
$D_2:=|R_1~R_2~R_3~R_4~R_5~|~C_4~C_5~C_6~C_7~C_8|= x_5(x_4^2-x_3x_5)^2$
\end{center}
of $\delta_2$. Now, $\{D_1,D_2\}  \subset I_{r_2}(\delta_2)$, and have distinct irreducible factors in $R$. Therefore, $\{D_1,D_2\}$ forms a regular sequence in $R$. Hence,  grade $I_{r_2}(\delta_2) \geq 2$. Now consider the minors
\begin{center}
$D_1 : = |R_4 ~ R_6 ~ R_8 ~|~ C_1 ~ C_2 ~ C_3| = x_2^3-x_1x_2x_3$\\
$D_2 : = |R_3 ~ R_4 ~ R_7 ~|~ C_1 ~ C_2 ~ C_3| = x_3^3-x_1x_3x_5$\\
$D_3 : = |R_1 ~ R_2 ~ R_5 ~|~ C_1 ~ C_2 ~ C_3| = x_4^3-x_3x_4x_5$
\end{center}
of $\delta_3.$ Now, with respect to the lexicographic monomial order induced by $x_2 > x_4 > x_3 > x_5 > x_1$ on $R,$ the leading terms of $D_1,D_2,D_3$ are mutually coprime. Therefore $\{D_1,D_2,D_3\}$ forms a regular sequence in $R.$ Hence, $I_{r_3}(\delta_2) \geq 3$. By the Theorem \ref{Buchsbaum-Eisenbud acyclicity criterion}, we conclude that $F_.$ is a graded free resolution of $k[S_{a,d,4}].$ Minimality of the resolution follows from the fact that all entries of $\delta_i's$ belong to the homogeneous maximal ideal $\textbf{m} = \langle x_1,x_2,x_3,x_4,x_5 \rangle$.
\end{proof}

\begin{proposition}\label{hilbertseriesk=4}
The Hilbert series $H(k[S_{a,d,4}];{\bf{t}})$ of $k[S_{a,d,4}]$ is given by
\[
 \frac{1 -{\bf{t}}^{2a+4d} - \sum_{i=2}^6 {\bf{t}}^{2a+id} + \sum_{i=4}^8{\bf{t}}^{3a+id}+ \sum_{i=5}^7{\bf{t}}^{3a+id} - \sum_{i=7}^9 {\bf{t}}^{4a+id}}{\prod_{i=0}^{4} (1- {\bf{t}}^{a+id})} .
\]
\end{proposition}
\begin{proof}
Let $\beta_{i,s}(k[S_{a,d,4}])$ denotes the $i^{\mathrm{th}}$ multigraded Betti number of $k[S_{a,k,4}]$ in degree $s$. Define
\[
C_i := \{\alpha:s  \mid s \in S_{a,d,4} ~ \& ~ \beta_{i,s}(k[S_{a,k,4}]) = \alpha (\neq 0) \} 
\]
Since $\mathrm{pdim}_R k[S_{a,d,4}] = 3,$ we have $C_i = \emptyset$ for $i > 3.$ Now by computing the degrees of syzygies from the Proposition \ref{resk=4}, we have $C_0:= \{{\bf{0}}\}$ and

\begin{center}
$C_1:= 
\left\lbrace
		\begin{array}{c}
			\{1:(2a+id) \mid i = 2,3,5,6\} 
			\cup \{ 2:(2a+4d)\}
		\end{array} 
		\right\rbrace $
\end{center}

\begin{center}
$C_2:= 
\left\lbrace
		\begin{array}{c}
		\{1:(3a+id) \mid i = 4,8 \}	\cup \{2:(3a+id) \mid i = 5,6,7\} 
		\end{array} 
		\right\rbrace $
\end{center}

\begin{center}
$C_3:= 
\left\lbrace
		\begin{array}{c}
		\{1:(4a+id) \mid i = 7,8,9 \}
		\end{array} 
		\right\rbrace $
\end{center}
Using \cite[Proposition, 8.23]{millersturmfels}, we can write
\[
H(k[S_{a,d,4}];{\bf{t}})= \frac{\sum_{i = 0,s \in S_{a,d,4}}^3 (-1)^i \beta_{i,s}(k[S_{a,d,4}]){\bf{t}}^s}{\prod_{i = 0}^4 (1-{\bf{t}}^{a+id})}
\]
Therefore with the help of $C_i$'s, we can write the Hilbert series
$H(k[S_{a,d,4}];{\bf{t}})$ as 
\[
 \frac{1 -{\bf{t}}^{2a+4d} - \sum_{i=2}^6 {\bf{t}}^{2a+id} + \sum_{i=4}^8{\bf{t}}^{3a+id}+ \sum_{i=5}^7{\bf{t}}^{3a+id} - \sum_{i=7}^9 {\bf{t}}^{4a+id}}{\prod_{i=0}^{4} (1- {\bf{t}}^{a+id})} .
\]
\end{proof}

Now, we recall the definition of Koszul rings. Let $I$ be a graded ideal in $k[x_1,\ldots,x_n]$. The ring $S = \frac{k[x_1,\ldots,x_n]}{I}$ is called Koszul if the minimal free resolution of $k$ over $S$ is linear, that is the entries in the matrices of the differentials are linear forms. It is known that if $I$ is generated by quadratic monomials then $S$ is Koszul. This leads to the following criterion of $S$ being Koszul using Gr\"{o}bner basis: 

\begin{theorem}\label{Koszul}
If $I$ has a quadratic Gr\"{o}bner basis then $S$ is Koszul.
\end{theorem}
\begin{proof}
See \cite[Theorem 34.12]{peeva}.
\end{proof}
\begin{theorem}
The ring $k[S_{a,d,k}]$ is Koszul.
\end{theorem}
\begin{proof}
From the Theorem \ref{Groebnerbasis}, we see that $I_{S_{a,d,k}}$ has a quadratic Gr\"{o}bner basis. Hence, the result follows from the Theorem \ref{Koszul}.
\end{proof}
\section{Castelnuovo–Mumford regularity of $I_{S_{a,d,k}}$}
Let $R= k[x_1,\ldots,x_n]$ and $I$ be a homogeneous ideal of $R$. Then the regularity of $I$ is defined as
\[
\mathrm{reg}(I) = \mathrm{max}_{0 \leq i \leq n-2} \{t_i - i\},
\] 
where $t_i$ is the maximum degree of the minimal $i$-syzygies of $I$ (see \cite{bayerm-regularity}).\\
Let $S = \langle a_1,\ldots,a_r,a_{r+1},\ldots,a_n\rangle$ be a simplicial affine semigroup with extremal rays $a_1,\ldots,a_r.$ Set $E = \{a_1,\ldots,a_r\}$. For any $b \in S$, consider the simplicial complex:
\[
T_b := \{F \subset E \mid b - a_F \in S\},
\]
where $a_F = \sum_{a \in F} a$ and $a_{\emptyset} = 0.$ Let $\tilde{H}_i(T_b)$ be the reduced $i$-th homology of the simplicial complex $T_b.$

\begin{lemma}\label{nonvanishinghomology}
Let $\Delta$ be a simplicial complex. Assume that $\tilde{H}_i(\Delta) \neq 0$ and let $c \in \tilde{Z}_i(\Delta) - \tilde{B}_i(\Delta),$ where $\tilde{Z}_i(\Delta)$ and $\tilde{B}_i(\Delta)$ are the spaces of cycles and boundaries respectively. Let $c = \sum_{j=1}^t \lambda_j F_j$, $\lambda_j \in k \setminus \{0\}$ for any $j = 1,\ldots,t$, $F_j \neq F_\ell$ if $j \neq \ell.$ Then if $F = \cup_{j=1}^t F_j$, one has that
\begin{center}
for all $p \in F$, there exist $q, ~1 \leq q \leq t$ such that  $F_q \cup \{p\} \notin \Delta$.
\end{center}
\end{lemma}
\begin{proof}
See \cite[Lemma 1.1]{briales-pison}.
\end{proof}

If $I_S$ is homogeneous ideal and $b \in S$ then $\|b\| = \alpha_1 + \ldots + \alpha_n$, where $b = \sum_{i=1}^n \alpha_i a_i$, is well defined. Define the set
\[
D(i) := \{ b \in S \mid \tilde{H}_i(T_b) \neq 0\}.
\]

\begin{theorem}\label{regularity}
With the above notations, assume that $I_S$ is homogeneous, then 
\[
\mathrm{reg}(I) = \mathrm{max}_{-1 \leq i \leq r-2} \{u_i - i\},
\]
where $u _i = \mathrm{max}\{\|b\| \mid b \in D(i)\}$.
\end{theorem}
\begin{proof}
see \cite[Theorem 16]{briales-campillo}.
\end{proof}

\begin{theorem}\label{regI_S_akd}
The Castelnuovo-Mumford regularity of $I_{S_{a,d,k}}$, $\mathrm{reg}(I_{S_{a,d,k}}) = 2.$
\end{theorem}
\begin{proof}
From the Theorem \ref{generatingset}, note that $I_{S_{a,d,k}}$ is homogeneous with respect to the standard grading on $k[x_1,\ldots,x_{k+1}].$ We claim that $D(0)= \emptyset.$ Suppose there exist $b \in S$ such that $\tilde{H}_0 (T_b) \neq 0.$ Since, $E = \{a, a+kd\}$, observe that the only possible choices of $T_b$ are
\begin{center}
$T_b = \{\emptyset, \{a\}\}$ or $ \{\emptyset, \{a+kd\}\}$ or $\{\emptyset, \{a\}, \{a+kd\}, \{a, a+kd\} \}.$
\end{center}
In each case, if there exist $0 \neq c \in \tilde{H}_0 (T_b)$, we get a contradiction to \ref{nonvanishinghomology}. Hence $D(0) = \emptyset$. Therefore by the Theorem \ref{regularity}, we have
\begin{center}
$\mathrm{reg}(I_{S_{a,d,k}}) = u_{-1}+1,$
\end{center}
where $u_{-1} = \mathrm{max}\{\|b\| \mid b \in D(-1)\}$. Note that $D(-1) = \mathrm{Ap}(S_{a,d,k}, E)$. Therefore, we conclude that
\begin{center}
$\mathrm{reg}(I_{S_{a,d,k}}) = \mathrm{max}\{\|b\| + 1 \mid b \in \mathrm{Ap}(S_{a,d,k}, E)\}.$ 
\end{center}
By the Lemma \ref{Aperyset}, we see that $\|b\| = 1,$ for all $b \in \mathrm{Ap}(S_{a,d,k}, E).$ Hence, we have  $\mathrm{reg}(I_{S_{a,d,k}}) = 2.$
\end{proof}

\begin{corollary}
The Castelnuovo-Mumford regularity of $k[S_{a,d,k}]$ is $1$.
\end{corollary}
\begin{proof}
Note that $\mathrm{reg}(I) = \mathrm{reg}(R/I) + 1$, for any homogeneous ideal $I$ of $R.$ Now the result follows from the Theorem \ref{regI_S_akd}.
\end{proof}

\section{An extension of $S_{a,d,k}$}
Let $b \in \mathbb{N}^2$ be such that $\mu b \in S_{a,d,k}$ for some $\mu > 0$ and $b \notin S_{a,d,k}$. Also, suppose that $\mu b = \sum_{i=1}^{k+1} \lambda_i (a+(i-1)d)$, where $\lambda_i \in \mathbb{N}$ such that either of $\lambda_1$ or $\lambda_{k+1}$ is non-zero. Define $S_{a,d,k}^b:= \langle a, a+d, \ldots, a+kd, b \rangle$. Let $R=k[x_1,\ldots,x_{k+1},y]$ and $\phi: k[x_1,\ldots,x_{k+1},y] \rightarrow k[t_1,t_2]$ such that $x_i \mapsto {\bf t}^{a+(i-1)d}$ for $i=1,\ldots,k+1$, $y \mapsto {\bf t}^b$. The semigroup ring $k[S_{a,d,k}^b]$ is isomorphic to the quotient ring $\frac{R}{I_{S_{a,d,k}^b}}$, where $I_{S_{a,d,k}^b}$ is the kernel of the algebra homomorphism $\phi$. These notation and assumptions will be followed throughout the section.

\begin{lemma}\label{simplicialb}
$S_{a,d,k}^b$ is a simplicial affine semigroup in $\mathbb{N}^2$ with respect to the extremal rays $a$ and $a+kd.$
\end{lemma}
\begin{proof}
Since $\mu b \in S_{a,d,k}$ and $S_{a,d,k}$ is simplicial with respect to the extremal rays $a$ and $a+kd$, there exist $\alpha (\neq 0) \in \mathbb{N}$ such that $\alpha(\mu b) = (\alpha \mu)b = \lambda_1 a + \lambda_2(a+kd)$ for some $\lambda_1, \lambda_2 \in \mathbb{N}.$ Now, it is immediate from Lemma \ref{simplicial}.
\end{proof}
 Now we recall the definition of Gluing of two subsemigroups of $\mathbb{N}^r$:
 \begin{definition}[{\cite[Theorem 1.4]{rosales97}}]
	Let $S \subseteq \mathbb{N}^r$ be an affine semigroup and $G(S)$ be the group spanned by $S$, that is, $G(S) =\{ a-b \in \mathbb{Z}^r \mid a, b \in S\}$. Let $A$ be the minimal generating system of $S$ and $A = A_1 \amalg A_2$ be a nontrivial partition of $A$ . Let $S_i$ be the submonoid of $\mathbb{N}^r$ generated by $A_i, i \in {1, 2}$. Then $S = S_1 + S_2.$ We say that $S$ is the gluing of $S_1$ and $S_2$ by $s$ if 
	\begin{enumerate}[{\rm(1)}]
		\item $s \in S_1 \cap S_2$ and,
		\item $G(S_1 ) \cap G(S_2) = s\mathbb{Z}$. 
	\end{enumerate}
\end{definition}

\begin{proposition}\label{Aperysetb}
Let $\mu$ be the smallest positive integer such that $\mu b \in S_{a,d,k}$. Then the Apery set of $S_{a,d,k}^b$ with respect to $E = \{a, a+kd \}$ is
\[
\mathrm{Ap}(S_{a,d,k}^b,E)= \left\lbrace 
	\begin{array}{c}
	\{(0,0)\}  \cup \{ b, 2b, \ldots, (\mu-1)b \}
	 \cup \{(a+id)+ \ell b \mid 1 \leq i \leq k-1 , 0 \leq \ell \leq \mu-1\} 
	\end{array}
	\right\rbrace
\]
\end{proposition}
\begin{proof}
Since $\mu b \in S_{a,d,k}$, we have $G(\{\mu b\}) \subseteq G(S_{a,d,k})\cap G(\{b\})$. Since $G(\{b\})$ is cyclic group, we have $G(S_{a,d,k})\cap G(\{b\})$ is also cyclic. Therefore, there exist $0 \neq \nu \in \mathbb{N}$ such that $G(\{\nu b\}) = G(S_{a,d,k})\cap G(\{b\})$. But, since $\mu$ be the smallest positive integer such that $\mu b \in S_{a,d,k}$, we have $G(\{\nu b\}) = G(S_{a,d,k})\cap G(\{b\}) = G(\{\mu b\})$. Therefore, $S_{a,d,k}^b$ is a gluing of $S_{a,d,k}$ and $\langle b \rangle$ by $\mu b$. Now, we have $\mu b = \sum_{i=1}^{k+1} \lambda_i (a+(i-1)d)$, where $\lambda_i \in \mathbb{N}$ such that either of $\lambda_1$ or $\lambda_{k+1}$ is non-zero. Therefore, by \cite[Theorem 1.4]{rosales97}, we have 
\[ 
I_{S_{a,d,k}^b} = I_{S_{a,d,k}} + \langle y^{\mu} - {\bf x}^{\lambda} \rangle, \quad \mathrm{where} \quad {\bf x} = x_1 \cdots x_{k+1}, \lambda = (\lambda_1, \ldots, \lambda_{k+1}).
\]
Now, define $R' = \frac{k[x_1,\ldots,x_{k+1},y]}{I_{S_{a,d,k}^b} + \langle x_1, x_{k+1} \rangle}$. Then $R'$ is generated by monomials as a $k$-vector space. Let $\mathfrak{B}$ be the set of monomial $k$-basis of $R'$. We have
\[
\mathfrak{B} = \left\lbrace 
	\begin{array}{c}
	\{1, x_2,\ldots,x_{k-1}, y,y^2,\ldots,y^{\mu-1}\}
	 \cup \bigcup_{i=2}^{k-1} \{x_i y^j \mid 1 \leq j \leq \mu-1\} 
	\end{array}
	\right\rbrace
\]
Now, by Theorem \cite[Theorem 3.3]{ojeda-tenorio}, we have 
\[
\mathrm{Ap}(S_{a,d,k}^b,E)= \{ \mathrm{deg}_{S_{a,d,k}^b}(u) \mid u \in \mathfrak{B}\}.
\]
Hence, we have 
\[
\mathrm{Ap}(S_{a,d,k}^b,E)= \left\lbrace 
	\begin{array}{c}
	\{(0,0)\}  \cup \{ b, 2b, \ldots, (\mu-1)b \}
	 \cup \{(a+id)+ \ell b \mid 1 \leq i \leq k-1 , 0 \leq \ell \leq \mu-1\} 
	\end{array}
	\right\rbrace .
\]
\end{proof}

\begin{corollary}\label{QF(S)b}
Let $\mu$ be the smallest positive integer such that $\mu b \in S_{a,d,k}$, then $\mathrm{QF}(S_{a,d,k}^b) = \{(\mu-1)b - (a + id) \mid 1 \leq i \leq k-1\}$.
\end{corollary}
\begin{proof}
From the Proposition \ref{Aperysetb}, we see that
\[
\max_{\preceq} \mathrm{Ap}(S_{a,d,k}^b,E) = \{(a+id)+(\mu-1)b \mid 1 \leq i \leq k-1 \}
\]
Therefore by lemma \ref{simplicialb}, we have
\[
\mathrm{QF}(S_{a,d,k}^b) = \{ -(a+ id) + (\mu-1)b \mid 1 \leq i \leq k-1\}.
\]
\end{proof}

\begin{corollary}
$k[S_{a,d,k}^b]$ is Cohen-Macaulay with type $k-1$. In Particular, $k[S_{a,d,k}^b]$ is Gorenstein if and only if $k=2$.
\end{corollary}
\begin{proof}
By Theorem \ref{CohenMacaulayness}, we know that $k[S_{a,d,k}]$ is Cohen-Macaulay and observe that $k[\langle b \rangle]$ is also Cohen-Mcaulay. From the proof of Proposition \ref{Aperysetb}, we have $S_{a,d,k}^b$ is a gluing of $S_{a,d,k}$ and $\langle b \rangle$. Therefore by \cite[Theorem 1.5]{hemagluingsemigroup}, we have  $k[S_{a,d,k}^b]$ is Cohen-Macaulay. From Corollary \ref{QF(S)b} and Theorem \ref{type}, it follows that Cohen-Macaulay type of $k[S_{a,d,k}^b]$ is $k-1$.
\end{proof}

From the proof of Proposition \ref{Aperysetb}, note that the extra minimal generator in $I_{S_{a,d,k}^b}$ is $f = y^{\mu} - {\bf x^{\lambda}}$. By \cite[Theorem 6.1]{hema2019}, we have that the mapping cone induced by the multiplication of $f$ to the minimal free resolution of $k[S_{a,d,k}]$ gives the minimal free resolution of $k[S_{a,d,k}^b]$. Therefore, we have the following:
\begin{corollary}\label{resb}
Let $R = k[x_1,\ldots,x_{k+1},y]$. Then\\
$(i)$ If $k=2$, then the complex
 \[
 0 \longrightarrow R \longrightarrow R^2 \longrightarrow R \longrightarrow R/I_{S_{a,d,k}^b} \longrightarrow 0,
\]  
is a minimal free resolution of $k[S_{a,d,2}^b]$.\\
$(ii)$ If $k=3$, then the complex
 \[
 0 \longrightarrow R^2 \longrightarrow R^5 \longrightarrow R^4 \longrightarrow R \longrightarrow R/I_{S_{a,d,k}^b} \longrightarrow 0,
\]  
is a minimal free resolution of $k[S_{a,d,3}^b]$.\\
$(ii)$ If $k=4$, then the complex
 \[
 0 \longrightarrow R^3 \longrightarrow R^{11} \longrightarrow R^{14} \longrightarrow R^7 \longrightarrow R \longrightarrow R/I_{S_{a,d,k}^b} \longrightarrow 0,
\]  
is a minimal free resolution of $k[S_{a,d,4}^b]$.
\end{corollary}

\section{Examples}
\begin{example}
Let $S = \langle (5,4), (9,13), (13,22), (17,31)\rangle$. Then for $a = (5,4)$, $d = (4,9)$ and $k = 3$, we have $S := S_{(5,4),(4,9),3}$. Therefore, by Theorem \ref{CohenMacaulayness} and Corollary \ref{Gorenstein}, we have that $k[S]$ is Cohen-Macaulay but not Gorenstein. By Corollary \ref{normal}, $k[S]$ is normal also. By Theorem \ref{generatingset}, we have 
\[
I_S = \langle x_2^2-x_1x_3, x_2x_3 - x_1x_4, x_3^2-x_2x_4 \rangle.
\]
Let $R = k[x_1,x_2,x_3,x_4].$ By Proposition \ref{resk=3}, we have 
\[
F_.:  0 \rightarrow R^2 \xrightarrow{\delta_2} R^3 \xrightarrow{\delta_1} R \rightarrow R/I_S \rightarrow 0
\]

is a minimal free resolution of $k[S].$ Also by Corollary \ref{hilbertseriesk=3}, we have 
\[
H(k[S],{\bf t}) = \frac{1-t_1^{18}t_2^{26}-t_1^{22}t_2^{35}-t_1^{26}t_2^{44}+t_1^{31}t_2^{48}+t_1^{35}t_2^{57}}{(1-t_1^5t_2^4)(1-t_1^9t_2^{13})(1-t_1^{13}t_2^{22})(1-t_1^{17}t_2^{31})}
\]
is the Hilbert series of $k[S].$ By Theorem \ref{Koszul}, $k[S]$ is Koszul also. Finally, by Theorem \ref{regI_S_akd}, the Castelnuovo-Mumford regularity of $I_S$ with respect to the standard grading on $R$ is $2$.
\end{example}

From the Corollary \ref{normal}, we know that $k[S_{a,d,k}]$ is always normal, but this may not be the case for $k[S_{a,d,k}^b]$.

\begin{example}
Let $S = \langle (2,3), (4,5), (6,7), (8,9), (9,11)\rangle$. Then for $a = (2,3)$, $d = (2,2)$, $b = (9,11)$ and $k = 3$, we have $S := S_{(2,3),(2,2),3}^{(9,11)}$ and $\mu=2$ is the smallest positive integer such that $\mu b = 2(2,3)+(6,7)+(8,9) \in S_{(2,3),(2,2),3}$. For $E = \{(2,3),(8,9)\}$, by Proposition \ref{Aperysetb}, we have
\[
\mathrm{Ap}(S,E) = \{(0,0),(4,5),(6,7),(13,16),(15,18)\}
\]
Therefore, \[\max_{\preceq} \mathrm{Ap}(S,E) = \{(13,16),(15,18)\}\] and hence, $\mathrm{QF}(S) = \{(3,4),(5,6)\}$. It is clear that $- ~ \mathrm{QF}(S) \nsubseteq \mathrm{relint}(S).$ Therefore by \cite[Theorem 4.6]{jafari-2}, $k[S]$ is not normal. From the proof of Proposition \ref{Aperysetb}, we have
\[
I_S=I_{S_{(2,3),(2,2),3}^{(9,11)}} = I_{S_{(2,3),(2,2),3}} + \langle y^{\mu} - {\bf x}^{\lambda} \rangle = \langle x_2^2-x_1x_3, x_2x_3 - x_1x_4, x_3^2-x_2x_4, y^2-x_1^2x_3x_4\rangle .
\]
For $R = k[x_1,x_2,x_3,x_4,y]$, By Corollary \ref{resb}, we have
\[
 0 \longrightarrow R^2 \longrightarrow R^5 \longrightarrow R^4 \longrightarrow R \longrightarrow R/I_S \longrightarrow 0,
\]  
is a minimal free resolution of $k[S]$.
\end{example} 

\bibliographystyle{plain}

\end{document}